\documentclass[ejsv2,noshowframe]{wlambda}
\makeatletter\def\journal@name{}\def\journal@issn{}\def\journal@url{}\makeatother

\usepackage{amsmath,amssymb,amsthm}
\RequirePackage[authoryear]{natbib}
\usepackage{graphicx}
\usepackage{booktabs}
\usepackage{longtable}
\usepackage{placeins}
\usepackage{hyperref}

\graphicspath{{Fig/}}

\arxiv{2511.09431}

\startlocaldefs
\theoremstyle{plain}
\newtheorem{theorem}{Theorem}[section]
\newtheorem{proposition}[theorem]{Proposition}
\newtheorem{lemma}[theorem]{Lemma}
\newtheorem{corollary}[theorem]{Corollary}

\theoremstyle{definition}
\newtheorem{definition}[theorem]{Definition}

\newtheorem{remark}[theorem]{Remark}

\newcommand{\conn}{\boldsymbol{C}}
\newcommand{\connaux}{\boldsymbol{\Lambda}}
\newcommand{\Reals}{\mathbb{R}}
\newcommand{\samFr}{\bar{\conn}_n}

\newcommand{\convPr}{\overset{P}{\rightarrow}}
\newcommand{\Mfd}{\mathcal{M}}

\newcommand{\Expe}{\mathbb{E}}
\newcommand{\SPD}{\operatorname{SPD}}
\newcommand{\norm}[1]{\left\lVert#1\right\rVert}

\DeclareMathOperator{\Vect}{Vec}

\newcommand{\code}[1]{\texttt{#1}}
\newcommand{\CRANpkg}[1]{\textsf{#1}}

\newcommand{\rev}[1]{#1}
\newcommand{\revcolor}{}
\newcommand{\revnote}[1]{}
\newenvironment{revblock}{\par}{\par}

\endlocaldefs

\begin{document}

\begin{frontmatter}
\title{A Riemannian Generalization of Wilks' Lambda for Testing Differences Among Brain Connectomes}
\runtitle{A Riemannian generalization of Wilks' Lambda}
\runauthor{N. Escobar and J. Harezlak}

\begin{aug}
\author[A]{\fnms{Nicolas}~\snm{Escobar}\ead[label=e1]{nescoba@upenn.edu}}
\author[B]{\fnms{Jaroslaw}~\snm{Harezlak}\ead[label=e2]{harezlak@iu.edu}}

\address[A]{Department of Biostatistics, Epidemiology and Informatics,
Perelman School of Medicine, University of Pennsylvania,
Philadelphia, PA 19104, USA\printead[presep={,\ }]{e1}}

\address[B]{Department of Epidemiology and Biostatistics,
Indiana University, 1025 E. 7th Street, Bloomington, IN 47405, USA\printead[presep={,\ }]{e2}}
\end{aug}

\begin{abstract}
Analysis of variance (ANOVA) for grouped samples whose observations take values in a space with nontrivial geometric structure, either a differentiable manifold or a non-Euclidean metric space, is a problem that has received attention in recent years. This paper is concerned with brain connectomes. Functional data analysis is another important motivation. Most of the literature relies on approaches that assume very little about the geometric structure. The most prominent example is Fr\'echet ANOVA, a method that is valid in metric spaces under relatively weak hypotheses. We argue that the generality of these approaches has a price in statistical power. We propose an ANOVA test on the space of symmetric positive definite matrices with the affine-invariant Riemannian metric that exploits features specific to that space. The test is not easy to generalize to other settings, but we show that there is a notable difference in power relative to Fr\'echet ANOVA. We show it theoretically, by characterizing the asymptotic behavior of both tests for alternatives close to the null hypothesis. We also show it in practice, by describing the specific consequences of that difference for the analysis of brain imaging studies whose design involves collecting samples at multiple sites.
\end{abstract}



\end{frontmatter}

\section{Introduction}\label{sec:intro}

The Human Connectome Project (HCP) is an ambitious set of brain imaging data collection initiatives whose collective goal is to characterize the normative trajectory of brain connectivity across the lifespan \citep{hcpya}. The project has been highly influential in recent neuroscience, in part because the acquisition protocol followed, with some variations, by its component initiatives has been widely adopted by similar projects. Its most influential component is the Young Adult study (HCP-YA), a sample of about $1{,}200$ subjects between $22$ and $35$ years of age that has been analyzed extensively in the literature. Another component, the Aging study (HCP-Aging), has a target enrollment of similar size and covers ages from $36$ to over $100$ years \citep{hcpa}.

Although it is one of the most important data sources in the field, the HCP poses methodological challenges for conclusions drawn from its data. The ideal scenario for a researcher interested in the full lifespan is one in which the HCP-YA and HCP-Aging data can be analyzed jointly. The challenge lies in the differences between the protocols of the two studies, which are small, as noted above, but not negligible. The most important is the scanner. HCP-Aging was carried out more recently and used newer scanners than HCP-YA. Other differences exist. Differences between the studies are also not the only concern, since there are issues within each study as well. HCP-YA was acquired at a single site, whereas HCP-Aging was acquired at four, which raises the question of whether data obtained at different sites are comparable.

These challenges have been studied in the literature, mostly within the framework of harmonization procedures. The most prominent is ComBat \citep{johnson2007}, together with its variants designed for brain imaging data \citep{fortin2018, yu}. This paper is not about harmonization procedures. It addresses an apparently simpler problem that comes before them, the detection of statistically significant differences between data from the same study collected at different sites, or between the two studies. We believe that better methods for this problem can inform the discussion of how effective harmonization procedures are. In brain connectomics that discussion has, in our view, not yet matured sufficiently.

Writing about the detection of significant differences between subgroups of a sample, in particular through analysis of variance (ANOVA), may seem dated in the era of artificial intelligence. In this setting it is less so, because connectome data have a complex structure that is worth reviewing. A connectome, the unit of data in this paper, is a symmetric positive definite matrix. We denote the set of such $p \times p$ matrices by $\SPD_p$. A connectome is obtained from the blood-oxygen-level-dependent signals recorded by functional magnetic resonance imaging at several regions of interest, by computing the correlation over time between each pair of signals. Such data can be vectorized in several ways. An extensive literature nonetheless shows important practical advantages in treating $\SPD_p$ as a Riemannian manifold, equipped with the affine-invariant Riemannian metric (AIRM), and therefore in treating connectome data as a sample taking values in a space with nontrivial geometric structure \citep{pennec, varo, abbas2021, goni}. Several alternatives have been proposed recently for the choice of metric \citep{logchol, airm1, airm3}. The AIRM is the one most used traditionally. It has theoretical properties that make it attractive, although it poses numerical difficulties that must be acknowledged.

Even so, ANOVA for samples of this kind is not a new topic. The most important method in the area is Fr\'echet ANOVA \citep{muller}, which has been widely adopted and has several important properties. Unlike traditional multivariate ANOVA (MANOVA), Fr\'echet ANOVA detects differences in both location and dispersion. Its popularity is due in part to its versatility. It applies in a very wide range of settings because it requires only the structure of a metric space, under relatively weak hypotheses on that space. Its statistics are also computed, for the most part, from squared distances between points of the space. They make little use of more complex constructions such as gradients, which makes them easy to implement and fast to compute.

Fr\'echet ANOVA is not, however, a generalization of MANOVA, in the following sense. If Euclidean space is regarded as a metric space and Fr\'echet ANOVA is applied to it, the usual MANOVA statistics are not recovered. This paper proposes a test that is a generalization in that sense. Specifically, it generalizes Wilks' Lambda \citep{anderson}. We construct matrices that quantify within-group and between-group variation, prove a sum-of-squares decomposition, and compute a statistic as a ratio of determinants.

This construction, which we call the Riemannian ANOVA, uses several tools available for $\SPD_p$ with the AIRM. It uses the tangent spaces of the manifold, the homogeneity of the space, and the explicit formulas for its curvature. It is therefore difficult to generalize to other settings. It is also computationally more expensive and detects differences in location only. As we show, however, there is an important difference in statistical power. We show it theoretically, in a simplified model in which all groups have the same dispersion and the same mean except one, whose mean lies at distance $\theta$ from the others. In that model the Fr\'echet ANOVA statistic has asymptotic behavior of order $\theta^4$, while that of the Riemannian ANOVA has order $\theta^2$. For small and moderate $\theta$ this favors the Riemannian ANOVA. We also show the difference in practice. Among the comparisons relevant to the HCP setting described above, there are some for which substantial indirect evidence indicates that the two samples differ systematically. In these comparisons the Riemannian ANOVA rejects the null hypothesis and Fr\'echet ANOVA fails to do so.

The method we propose does not consist of projecting the sample to a chosen tangent space, vectorizing, and then carrying out MANOVA. That procedure, which we call tangent-space MANOVA, is entirely natural in this setting. As we show, it does not offer the same power advantages as the Riemannian ANOVA.

The aim of this paper is not exactly to propose a new ANOVA method for connectomics that is better than Fr\'echet ANOVA. In many of the respects detailed above it is not necessarily better. The aim is to demonstrate a more general point with a concrete example. For samples taking values in spaces with geometric structure, exploiting the particular features of that structure is worthwhile. Some generality is lost and statistical power is gained. Exploiting the structure also has a computational cost. That cost keeps falling, which leaves the question of the best approach more open.

The remainder of the paper is organized as follows. Section~\ref{sec:wilks} presents the model and the statistic. Section~\ref{sec:asymptotics} develops the asymptotic theory and the power comparison with Fr\'echet ANOVA. Section~\ref{sec:permutation} describes the permutation calibration used throughout. Section~\ref{sec:simulation} reports two simulation studies, the first with the three tests and the second against tangent-space MANOVA at two base points. Section~\ref{sec:application} presents the connectome application. Section~\ref{sec:discussion} concludes. Technical material, the model diagnostics and the robustness of the application to the connectivity estimator are collected in the appendices. The per-cell tables of the simulation studies and the code are in the Supplementary Material.

\section{A Riemannian Generalization of Wilks' Lambda}\label{sec:wilks}

\subsection{The Riemannian normal model}\label{sec:model}

Throughout this paper, $\Mfd = \SPD_p$ denotes the set of $p \times p$ symmetric positive definite matrices. Since it is an open subset of the set $\mathrm{Symm}_p$ of symmetric matrices, $\Mfd$ is a differentiable manifold of dimension $d = p(p+1)/2$. Each of its tangent spaces $T_{\conn}\Mfd$, with $\conn \in \Mfd$ a generic point, is naturally isomorphic to $\mathrm{Symm}_p$. We equip $\Mfd$ with the affine-invariant Riemannian metric (AIRM), which makes it a Hadamard manifold. The details of the definition of this metric are given in \citet{pennec}. On such manifolds the Fr\'echet mean of a random element $\conn$ with finite second moment, defined as
\begin{equation}\label{eq:frechet-mean}
\operatorname{argmin}_{\connaux \in \Mfd} \; \Expe \norm{\log_{\connaux}(\conn)}_{\connaux}^{2},
\end{equation}
exists and is unique \citep{kendall}. In this display $\log_{\conn} : \Mfd \rightarrow T_{\conn}\Mfd$ is the Riemannian logarithm and $\norm{\cdot}_{\conn}$ is the norm of the metric on $T_{\conn}\Mfd$, so that $\norm{\log_{\connaux}(\conn)}_{\connaux}$ is the geodesic distance $d(\connaux, \conn)$. The logarithm is the basis of the vectorization
\begin{align}
	\Vect_{I}: T_I\Mfd &\rightarrow \Reals^{p(p+1)/2} \nonumber \\
		\lbrace w_{ij}\rbrace &\mapsto (w_{11}, \sqrt{2}w_{12}, w_{22}, \ldots )
\end{align}
and
\begin{align}
	\Vect_{\conn}: T_{\conn}\Mfd &\rightarrow \Reals^{p(p+1)/2} \nonumber \\
		\log_{\conn}(\connaux) &\mapsto \Vect_I(\conn^{-1/2} \log_{\conn}(\connaux) \conn^{-1/2}). \label{eq:vecC}
\end{align}
With this vectorization, and for a point $\conn^\ast \in \Mfd$ and a positive definite matrix $\Gamma$, we define the distribution
\begin{align}
	p(\conn) \propto
	\exp\left(-\frac{1}{2}
	\Vect_{\conn^\ast}(\log_{\conn^\ast}(\conn))^T
	\Gamma^{-1}
	\Vect_{\conn^\ast}(\log_{\conn^\ast}(\conn))\right), \label{eq:riemn-density}
\end{align}
which we call the Riemannian normal model and denote by $\conn \sim \mathrm{RiemN}(\conn^\ast, \Gamma)$. Its Fr\'echet mean is $\conn^\ast$, as shown in Appendix~\ref{sec:supp-uniqueness}. We call $\Gamma$ its tangent covariance. This model generalizes the one described by \citet{varo}. Appendix~\ref{sec:supp-diagnostics} shows that the generalization is needed, since the model of \citet{varo} assumes an isotropic tangent covariance, which the HCP-YA connectomes do not satisfy.

\subsection{Test statistic and sum-of-squares decomposition}\label{sec:statistic}

Our main contribution is a statistic that generalizes Wilks' Lambda to samples taking values in $\Mfd$. Consider $\Mfd$-valued random variables $C_{ l j }$, $l = 1, \ldots, g$, $j = 1, \ldots, n_l$, and let $n = \sum_{l=1}^g n_l$. Let $\hat{C}$ denote the sample Fr\'echet mean of the whole sample, that is, the Fr\'echet mean of its empirical distribution, and let $\hat{C}_{ l }$ denote the sample Fr\'echet mean of the subsample $\lbrace C_{ l j } \rbrace_j$. Like the original, the statistic starts from a sum-of-squares decomposition.

\begin{proposition}
\label{prop:test-statistic}
Let $v_l = \operatorname{Vec}_{ \hat{C}_{ l } }( \log_{ \hat{C}_{ l } } \hat{C} )$, $u_{l,j} = \operatorname{Vec}_{ \hat{C}_{ l } }( \log_{ \hat{C}_{ l } } C_{ l j } )$, and $w_{l,j} = v_l - u_{l,j}$. Define
\begin{align}
T &= \sum_{l=1}^g\sum_{j=1}^{n_l} w_{lj}w_{lj}^T, \\
W &= \sum_{l=1}^g  \sum_{j=1}^{n_l} u_{lj}u_{lj}^T.
\end{align}
If $T$ is positive definite, then the test statistic
\begin{equation}\label{eq:wilks}
\Lambda^\ast = \frac{|W|}{|T|}
\end{equation}
ranges between 0 and 1.
\end{proposition}

\begin{proof}
Since the vectors $\lbrace u_{lj} \rbrace_j$ are centered (i.e., $\sum_{j=1}^{n_l} u_{lj} = 0$), we have
\begin{gather}
\sum_{j=1}^{n_l} w_{lj}w_{lj}^T = \sum_{j=1}^{n_l}(v_l - u_{lj})(v_l - u_{lj})^T \nonumber \\
= n_l v_l v_l^T + \sum_{j=1}^{n_l} u_{lj}u_{lj}^T. \label{eq:anova-decomp}
\end{gather}
Adding over $l$ yields the decomposition
\begin{equation}
T = \sum_{l=1}^g n_lv_lv_l^T + W.
\label{orth}
\end{equation}
Both matrices are positive semidefinite and equation \eqref{orth} shows $T\geq W$, hence $0 \leq \Lambda^\ast \leq 1$.
\end{proof}

\begin{remark}\label{rem:tangent-manova}
Notice the difference between $w_{lj}$ and the vector $\operatorname{Vec}_{ \hat{C} }( \log_{ \hat{C} } C_{ l j } )$. In general, they are not the same, but in the Euclidean case they coincide, so the subsequent developments reduce to the traditional MANOVA in that case.
\end{remark}

We consider the null hypothesis $H_0: C_1^* = C_2^* = \cdots = C_g^*$, where $C_l^*$ denotes the population Fr\'echet mean of group $l$. \rev{Table~\ref{tab:notation} collects the notation used in the paper.}

\begin{table}[t]
\revcolor
\centering
\caption{Notation used in the paper, in order of introduction.}
\label{tab:notation}
\begin{tabular}{@{}l p{7.4cm} l@{}}
\toprule
Symbol & Meaning & Defined \\
\midrule
$\Mfd = \SPD_p$ & $p \times p$ symmetric positive definite matrices & Section~\ref{sec:model} \\
$d = p(p+1)/2$ & dimension of $\Mfd$ and of its tangent spaces & Section~\ref{sec:model} \\
$T_{\conn}\Mfd$, $\log_{\conn}$ & tangent space at $\conn$; Riemannian logarithm at $\conn$ & Section~\ref{sec:model} \\
$d(\cdot,\cdot)$, $\norm{\cdot}_{\conn}$ & geodesic distance; norm on $T_{\conn}\Mfd$ & Section~\ref{sec:model} \\
$\Vect_{\conn}$ & isometric vectorization $T_{\conn}\Mfd \to \Reals^{d}$, \eqref{eq:vecC} & Section~\ref{sec:model} \\
$\conn^\ast$, $\Gamma$ & population Fr\'echet mean \eqref{eq:frechet-mean} and tangent covariance & Section~\ref{sec:model} \\
$\mathrm{RiemN}(\conn^\ast, \Gamma)$ & Riemannian normal model \eqref{eq:riemn-density} & Section~\ref{sec:model} \\
$g$; $n_l$, $n$; $\lambda_l = n_l/n$, $\pi_l$ & number of groups; group and total sample sizes; group proportions and their limits & Sections~\ref{sec:statistic}--\ref{sec:asymptotics} \\
$C_l^\ast$, $\hat{C}_l$; $\hat{C}$ & population and sample Fr\'echet means of group $l$; sample Fr\'echet mean of the pooled sample & Section~\ref{sec:statistic} \\
$u_{lj}$, $v_l$, $w_{lj}$ & residual, group displacement and total deviation, all in $T_{\hat{C}_l}\Mfd$ & Proposition~\ref{prop:test-statistic} \\
$W$, $T$; $\Lambda^\ast = |W|/|T|$ & within-group and total scatter matrices; Riemannian Wilks' Lambda \eqref{eq:wilks} & Proposition~\ref{prop:test-statistic} \\
$T_n$; $U_n$, $F_n$; $\hat{V}_l$, $\hat{V}_p$ & Fr\'echet ANOVA statistic \eqref{eq:frechet-stat}; its scale and location components; group and pooled sample Fr\'echet variances & Section~\ref{sec:frechet-comparison} \\
$\delta$, $\delta_{\text{Fr\'echet}}$; $\theta$ & asymptotic non-centralities; Riemannian size of a location shift & Section~\ref{sec:frechet-comparison} \\
$B$, $\hat{p}$ & number of permutations; permutation $p$-value \eqref{eq:perm-pvalue} & Section~\ref{sec:permutation} \\
$\samFr$, $\widehat{\Gamma}$ & sample Fr\'echet mean and shrinkage estimate of the tangent covariance of the calibration cohort & Section~\ref{sec:simulation} \\
$x_{lj}$, $\bar{x}_l$, $\bar{x}$; $W_E$, $T_E$, $\Lambda_E$ & tangent-space MANOVA: projected vectors at $\hat{C}$ and their group and overall means; Euclidean scatter matrices and Wilks' Lambda \eqref{eq:tangent-manova} & Section~\ref{sec:simulation} \\
\bottomrule
\end{tabular}
\end{table}

\section{Asymptotic Theory}\label{sec:asymptotics}

With $\Lambda^\ast$ defined, a test of the null hypothesis can be constructed by determining the asymptotic distribution of the statistic under that hypothesis and the Riemannian normal model, and by using that distribution to set critical values. The strategy for determining the distribution is conceptually simple. If all population means are equal, and therefore equal to the overall population mean, then the sample means approach the overall population mean as the sample size grows. It then suffices to consider a small neighborhood of that mean. Every sufficiently small neighborhood in a manifold looks Euclidean, so the asymptotic behavior of the Riemannian Wilks' Lambda approaches that of the classical statistic.

The asymptotic behavior of the Fr\'echet ANOVA statistic, in turn, is known \citep{muller}. Schematically, that statistic is built from Fr\'echet variances, which are averages of squared distances between the observations and the sample means. The variances enter the statistic through the square of a difference. The statistic therefore depends approximately quartically on the deviation from the null hypothesis. The two statistics can then be compared. The analysis described above leads to the conclusion that the negative logarithm of the Riemannian Wilks' Lambda depends quadratically on the same deviation, so the proposed test has an advantage for small deviations. The rest of this section gives the details of this argument.

\subsection{Distribution under the null hypothesis}\label{sec:null}

Under normality, the asymptotic distribution of the negative logarithm of the statistic can be obtained.

\begin{theorem}
\label{thm:asymptotic}
Suppose that the $C_{lj}$, $l = 1, \ldots, g$, $j = 1, \ldots, n_l$, are independent with $C_{lj} \sim \mathrm{RiemN}(C^\ast, \Gamma)$, so that all population group means equal $C^\ast$. Let $n = \sum_{l=1}^g n_l \to \infty$ in such a way that $n_l/n$ converges to a fixed $\pi_l \in (0,1)$ for each $l$. Then
\begin{equation}-n\log\Lambda^* \xrightarrow{d} \chi^2_{d(g-1)},\end{equation}
where $d = p(p+1)/2$ is the dimension of $\Mfd$.
\end{theorem}

The idea of the proof was described above. The details are in Appendix~\ref{sec:supp-proofs}. Two technical points deserve mention. First, the Riemannian normality assumption implies multivariate normality of the projections to the tangent space at the population mean, not of the projections to the tangent space at the sample mean. The argument of \citet{clt4} establishes the latter normality asymptotically. Second, once that fact is established, the argument proceeds as in the Euclidean case \citep{anderson}, except that deviations due to curvature must be bounded. The bound follows from a formula for differences of logarithms, whose details are in Appendix~\ref{sec:supp-proofs}.

\subsection{Power comparison with Fr\'echet ANOVA}\label{sec:frechet-comparison}

The Fr\'echet ANOVA statistic \citep{muller} has the form
\begin{equation}\label{eq:frechet-stat}T_n = \frac{n U_n }{\sum_{l=1}^{g}\frac{\lambda_{l,n}}{ \hat{\sigma}_l^2}} + \frac{nF_n^2}{\sum_{l=1}^{g}\lambda_{l,n}^2\hat{\sigma}_l^2}\end{equation}
where $\lambda_l = n_l/n$ and $F_n = \hat{V}_p - \sum_{l=1}^{g}\lambda_l \hat{V}_l$ measures between-group variation, $U_n = \sum_{j < l} \frac{\lambda_j\lambda_l}{\hat{\sigma}_j^2 \hat{\sigma}_l^2} (\hat{V}_j-\hat{V}_l)^2$ measures variance heterogeneity, $\hat{V}_l = \frac{1}{n_l}\sum_{j=1}^{n_l} d^2(C_{lj}, \hat{C}_l)$ is the sample Fr\'echet variance for group $l$, $\hat{V}_p = \sum_{l=1}^{g}\lambda_l \frac{1}{n_l}\sum_{j=1}^{n_l} d^2(C_{lj}, \hat{C})$ is the overall sample Fr\'echet variance, and $\hat{\sigma}_l^2 = n_l^{-1}\sum_{j=1}^{n_l} d^4(C_{lj}, \hat{C}_l) - \hat{V}_l^2$ is the sample variance of the squared distances in group $l$. Under $H_0$, $T_n \xrightarrow{d} \chi^2_{g-1}$ as $n \to \infty$.

The form of this statistic as a sum of two contributions shows that the proposed test and Fr\'echet ANOVA are not strictly comparable. The proposed test detects only differences in location between the groups, whereas Fr\'echet ANOVA has an additional term designed to detect differences in dispersion. None of the arguments that follow suggests that the proposed test is universally more powerful than Fr\'echet ANOVA, or even that it is preferable in general. They establish only that there is a specific regime in which Fr\'echet ANOVA is not optimal and the proposed test performs better. Later sections show that this regime is of practical relevance for brain connectomics. The regime is the Riemannian normal model with equal covariances.

Before stating the propositions of this section, we describe their purpose. The propositions do not give an explicit comparison of the power functions of the two tests. Later sections carry out that comparison empirically, with simulations and with experiments on real data. The purpose of the propositions is to give an argument that makes those empirical findings plausible. As a by-product, the argument clarifies some properties of the proposed test statistic. The most important is that its behavior under a fixed alternative has a simple geometric meaning related to the Mahalanobis distance.

This plausibility argument uses a term that we call the leading-order non-centrality, whose technical definition is in Appendix~\ref{sec:supp-proofs}. In short, we describe what happens under the alternative hypothesis but infinitesimally close to the null. However close the alternative is to the null, once its specific form is fixed the statistic of each test grows without bound with the sample size, in mean or in probability. After rescaling by the sample size, the statistic converges to a limit that depends on how far the alternative is from the null. The non-centrality is this limit multiplied by the sample size. If the distance to the null is then made small, the limit can be expanded in a Taylor series. The term ``leading-order'' refers to the first term of that expansion. The intuition is simple. We are interested in the regime of large sample size and small effect size.

\begin{proposition}[Fr\'echet ANOVA Non-centrality Under Normality]\label{prop:frechet-noncentrality}
Under the Riemannian normal model with equal covariances and a fixed alternative $H_1$, \rev{the leading-order shift of the Fr\'echet ANOVA statistic $T_n$ from its null behavior is asymptotically}
\begin{equation}\delta_{\text{Fr\'echet}} = \frac{n \left(\sum_{l=1}^g \lambda_l d^2(C_l^*, C^*)\right)^2}{2\operatorname{tr}(\Gamma^2) \sum_{l=1}^g \lambda_l^2},\end{equation}
with $C^*$ denoting the overall population Fr\'echet mean. Here $2\operatorname{tr}(\Gamma^2)$ is the variance of $d^2(C_{lj}, C_l^*)$, the limit of each $\hat{\sigma}_l^2$. The shift stays bounded along alternatives with $d(C_l^*, C^*)$ of order $n^{-1/4}$ and vanishes along alternatives of order $n^{-1/2}$.
\end{proposition}

\begin{proposition}[Non-centrality under the Alternative Hypothesis]
\label{prop:noncentrality}
Under the alternative hypothesis $H_1$ with population Fr\'echet means $C_1^*, \ldots, C_g^* \in \rev{\SPD_p}$ where at least two differ, let $C^*$ denote the overall population Fr\'echet mean and define the population tangent vectors $v_l^* = \operatorname{Vec}_{C^*}(\log_{C^*}(C_l^*))$ for $l = 1, \ldots, g$. Then, with $n_l/n \to \pi_l > 0$, the statistic $-n\log\Lambda^*$ has leading-order non-centrality
\begin{equation}\delta = n \sum_{l=1}^g \pi_l (v_l^*)^T (W^*)^{-1} v_l^*,\end{equation}
where $W^*$ is the population within-group covariance matrix in the tangent space; that is, $-n\log\Lambda^* \approx \delta$ to leading order. This is the sum of squared Mahalanobis distances of the group means from the overall mean in the tangent space. Consequently the power at any fixed level $\alpha$ tends to $1$ as $n \to \infty$.
\end{proposition}

The two preceding propositions are abstract, which motivates the next one. Their proofs are in Appendix~\ref{sec:supp-proofs}. The next proposition examines a specific and common case of the setting described in their hypotheses. That case allows a direct comparison and gives a more concrete result that helps interpret the two propositions. We do not claim that this specific case is of practical interest or that it is realistic. Its purpose is purely illustrative.

\begin{proposition}[Comparison of Non-centrality Parameters]
\label{prop:noncentrality_comparison}
Consider the Riemannian ANOVA and Fr\'echet ANOVA tests under the Riemannian normal model
with isotropic covariance $\Gamma = \sigma^2 \rev{I_{d}}$
and balanced design ($n_l = n/g$ for all groups).
For a specific alternative in which groups $2, \ldots, g$ share a mean $C_0^*$ and group 1 lies at Riemannian distance $\theta > 0$ from it, the non-centrality parameters are, to leading order in $\theta$,
\begin{align}
\delta_{\text{Riemannian}} &= \frac{n(g-1)\theta^2}{g^2\sigma^2} = O(\theta^2),\\
\delta_{\text{Fr\'echet}} &= \frac{n(g-1)^2\theta^4}{2g^3\sigma^4 d} = O(\theta^4).
\end{align}
\end{proposition}

\begin{proof}
To leading order in $\theta$ the overall mean $C^*$ lies on the geodesic from $C_0^*$ to $C_1^*$ at distance $\theta/g$ from $C_0^*$, so that $\|v_1^*\| = d(C_1^*, C^*) = (g-1)\theta/g$ and $\|v_l^*\| = d(C_l^*, C^*) = \theta/g$ for $l \geq 2$, and $\sum_l \lambda_l \|v_l^*\|^2 = \sum_l \lambda_l d^2(C_l^*, C^*) = (g-1)\theta^2/g^2$.
The general expression for the Riemannian non-centrality parameter (Proposition~\ref{prop:noncentrality}) can be rewritten as $\delta_{\text{Riemannian}} = n \sum_{l=1}^g \lambda_l \|v_l^*\|_{\Gamma^{-1}}^2$, where $v_l^* = \operatorname{Vec}(\log_{C^*}(C_l^*))$ and $\lambda_l = 1/g$ for balanced designs. With $\Gamma = \sigma^2 I_d$,
\begin{equation}\delta_{\text{Riemannian}} = n\sigma^{-2}\sum_{l=1}^g \lambda_l\|v_l^*\|^2 = \frac{n(g-1)\theta^2}{g^2\sigma^2}.\end{equation}
On the other hand, the Fr\'echet ANOVA non-centrality parameter of Proposition~\ref{prop:frechet-noncentrality} is
\begin{equation}\delta_{\text{Fr\'echet}} = \frac{n \left(\sum_{l=1}^g \lambda_l d^2(C_l^*, C^*)\right)^2}{2\operatorname{tr}(\Gamma^2) \sum_{l=1}^g \lambda_l^2}.\end{equation}
Since $\operatorname{tr}(\Gamma^2) = \sigma^4 d$ and $\sum_l \lambda_l^2 = 1/g$,
\begin{equation}\delta_{\text{Fr\'echet}} = \frac{n\{(g-1)\theta^2/g^2\}^2}{2\sigma^4 d \cdot (1/g)} = \frac{n(g-1)^2\theta^4}{2g^3\sigma^4 d}.\end{equation}
\end{proof}

\section{Inference by Permutation}\label{sec:permutation}

The asymptotic reference of Theorem~\ref{thm:asymptotic} requires $n$ to be large relative to the tangent dimension $d = p(p+1)/2$, a condition connectome applications do not \rev{typically meet. At} $p = 116$ the tangent dimension is $6{,}786$, against cohorts \rev{on the order of several} hundred subjects. \rev{For that reason, all} inference reported below is by permutation.

\rev{Suppose all groups are identically distributed. Then group labels are exchangeable.} Let $\Lambda^*_{\mathrm{obs}}$ denote the statistic computed from the observed labeling. We draw $B$ relabelings uniformly at random from those preserving the group sizes $n_1, \ldots, n_g$ and, for each, recompute the group Fr\'echet means, the scatter matrices $W$ and $T$, and hence the statistic, obtaining $\Lambda^{*(1)}, \ldots, \Lambda^{*(B)}$. Since $\Lambda^*$ takes values in $[0,1]$ and decreases as the groups separate, large values of $-n\log\Lambda^*$ are evidence against $H_0$, and we report the proportion of permuted statistics exceeding the observed one,
\begin{equation}
\hat{p} = \frac{1}{B} \, \#\left\{ b : -n\log\Lambda^{*(b)} > -n\log\Lambda^*_{\mathrm{obs}} \right\},
\label{eq:perm-pvalue}
\end{equation}
rejecting when $\hat{p} \leq \alpha$; $\hat{p}$ therefore takes values in multiples of $1/B$\rev{. An observed statistic that exceeds every permuted value is reported as $p < 1/B$. The smallest level at which a rejection can occur is $1/(B+1)$, the probability under $H_0$ that the observed statistic is the largest of the $B+1$ exchangeable values, so the size attained at a nominal level $\alpha$ is $(\lfloor \alpha B \rfloor + 1)/(B+1)$ rather than $\alpha$; this matters at the small values of $B$ that the cost of the test imposes and is taken into account in Section~\ref{sec:simulation}}. The number of permutations is stated with each study; the cost of a draw is dominated by the two $d \times d$ determinants. The same scheme is applied to Fr\'echet ANOVA \rev{in the simulation study, where the two tests are compared on equal terms; in the application of Section~\ref{sec:application}, Fr\'echet ANOVA is referred to its asymptotic $\chi^2_{g-1}$ distribution, which for that scalar statistic does not face the $d \gg n$ obstruction above}.

In some situations the computation of the proposed statistic, and of the tangent-space MANOVA statistic, requires numerical care related to the conditioning of the matrices $W$ and $T$. Specifically, it requires a regularization step, whose details are in Appendix~\ref{sec:supp-computation}.

\section{Simulation Studies}\label{sec:simulation}

\subsection{Power against Fr\'echet ANOVA at connectome dimensions}\label{sec:sim-frechet}

We report two simulation studies. The first, in this subsection, has three objectives. The first is to verify that the permutation tests hold their attainable Type~I error. The second is to examine the power comparison of Section~\ref{sec:frechet-comparison} in a controlled setting. The third is to compare the proposed statistic with tangent-space MANOVA, defined below. The second study, in Section~\ref{sec:sim-basepoint}, takes that comparison to dimensions at which the scatter matrices have full rank. Data are generated from a Riemannian normal law whose parameters are fitted to \rev{the Human Connectome Project Young Adult (HCP-YA) cohort in the Schaefer-100 plus Tian-S1 parcellation} ($p = 116$, $940$ subjects)\rev{. We estimate} the sample Fr\'echet mean $\samFr$ and the tangent covariance $\widehat{\Gamma}$, the latter regularized by oracle-approximating shrinkage \citep{chen2010} toward a scaled identity, which keeps $\widehat{\Gamma}$ positive definite when $d$ exceeds $n$. To resolve how the tests behave as the tangent dimension grows, we report a dense dimension series at $p \in \{25, 30, 35, 40\}$ (tangent dimension $d = 325, 465, 630, 820$), each calibrated to genuine connectivity sub-structure by restricting every pooled connectome to its leading $p \times p$ principal submatrix and refitting $(\samFr, \widehat{\Gamma})$ at that dimension. \rev{This grid stops at $p = 40$ for the computational reason given at the end of this section.} Inference is by permutation throughout, as described in Section~\ref{sec:permutation}, with $B = 49$\rev{, the largest number of permutations that the budget of the full design --- $24$ null and $44$ power configurations, each of $R = 200$ replicates --- allowed; $\hat{p}$ is then resolved to $1/49$, and the sizes attainable at the nominal levels $\alpha = 0.01, 0.05, 0.10$ are $0.02$, $0.06$ and $0.10$}; we discuss the application dimension $p = 116$ ($d = 6{,}786$) below.

\emph{The three tests.}\enspace Each replicate is analyzed with three tests. The first is the Riemannian ANOVA of Section~\ref{sec:wilks}, the statistic $-n\log\Lambda^*$. The second is Fr\'echet ANOVA, the statistic $T_n$ of \eqref{eq:frechet-stat}. The third is tangent-space MANOVA. Every observation is mapped to the tangent space at the pooled sample Fr\'echet mean, $x_{lj} = \Vect_{\hat{C}}(\log_{\hat{C}} C_{lj})$, and classical Wilks' Lambda is formed from the Euclidean scatter matrices of the $x_{lj}$,
\begin{equation}\label{eq:tangent-manova}
\Lambda_{E} = \frac{|W_E|}{|T_E|}, \qquad
W_E = \sum_{l=1}^{g}\sum_{j=1}^{n_l} (x_{lj} - \bar{x}_l)(x_{lj} - \bar{x}_l)^T, \qquad
T_E = \sum_{l=1}^{g}\sum_{j=1}^{n_l} (x_{lj} - \bar{x})(x_{lj} - \bar{x})^T,
\end{equation}
where $\bar{x}_l$ and $\bar{x}$ are the Euclidean means of the projected vectors within group $l$ and overall. The base point $\hat{C}$ is the choice most favorable to this procedure, since it is the point about which the projected data are centered. The matrices $W_E$ and $T_E$ pass through the same regularization step as $W$ and $T$ (Appendix~\ref{sec:supp-computation}), so the two determinant ratios are compared on equal terms. Since $\hat{C}$ is invariant under relabeling, the vectors $x_{lj}$ are computed once per replicate and only the group labels are permuted. The first two tests share one set of $B$ relabelings per replicate. The third draws its own set of $B$ relabelings. All three use the $p$-value convention \eqref{eq:perm-pvalue}.

\emph{Type~I error.}\enspace The empirical rejection rates of the three tests under the null \rev{(both groups drawn from the fitted Riemannian normal law} $\mathrm{RiemN}(\samFr, \widehat{\Gamma})$ \rev{of Section~\ref{sec:model}, equation~\eqref{eq:riemn-density}}, across per-group sample sizes $n_l \in \{20,\allowbreak 50,\allowbreak 100,\allowbreak 150,\allowbreak 205,\allowbreak 213\}$ and nominal levels $\alpha \in \{0.01,\allowbreak 0.05,\allowbreak 0.10\}$, with $R = 200$ replicates per \rev{cell) are} reported in full in the Supplementary \rev{Material. Each} entry, a binomial proportion over the $R$ replicates, carries a Wilson $95\%$ confidence interval. All three tests hold their attainable level. Averaged over the $24$ cells at each nominal level, the empirical rates are $0.020$, $0.062$ and $0.097$ for the Riemannian ANOVA, $0.020$, $0.064$ and $0.105$ for Fr\'echet ANOVA, and $0.018$, $0.061$ and $0.096$ for tangent-space MANOVA, against the attainable sizes $0.02$, $0.06$ and $0.10$ of $B = 49$ permutations. Of the $72$ cells per test, $71$, $67$ and $71$ fall within two Monte-Carlo standard errors of the attainable size, with the standard error computed at that size. No empirical rate at nominal $\alpha = 0.05$ exceeds $0.095$ for any test. The exceedances show no pattern in $n_l$ or in $p$. The asymptotic $\chi^2_{d(g-1)}$ reference, by contrast, is severely anti-conservative wherever $n$ does not greatly exceed $d$\rev{, so the permutation calibration of Section~\ref{sec:permutation} is the appropriate inference in this regime}. %

\emph{Power.}\enspace For the power study, one group is shifted from $\samFr$ along a fixed unit tangent direction to Riemannian distance $\theta$ while the other remains at $\samFr$, with $n_l = 100$ per group and $n = 200$ in total; $\theta$ ranges over $[0, 2]$, the band on which the power curves separate before all three saturate. Figure~\ref{fig:sim-power} plots the power of the three tests against $\theta$, faceted by dimension, with Wilson $95\%$ confidence intervals shown as error bars. Two features of the figure carry the conclusions of this section.

The first is the separation between the two tangent-space tests and Fr\'echet ANOVA. The Riemannian ANOVA and tangent-space MANOVA reach power one half between $\theta = 0.4$ and $\theta = 0.6$ at every $p$. Fr\'echet ANOVA remains close to its null level up to $\theta = 0.6$ at every $p$. It reaches power one half only between $\theta = 1.0$ and $1.2$ at $p = 25$ and $30$, between $1.2$ and $1.4$ at $p = 35$, and between $1.4$ and $1.6$ at $p = 40$. The separation is \rev{clear} relative to the confidence intervals throughout the rising region. As $d$ grows, the $1/d$ factor in the Fr\'echet non-centrality of Proposition~\ref{prop:noncentrality_comparison} flattens the Fr\'echet power curve. \rev{At} $p = 40$ ($d = 820$) Fr\'echet ANOVA barely responds over much of the band while the Riemannian test reaches near-certain power by $\theta = 1.2$. The advantage is expected to persist at the $p = 116$ parcellation of the application, where $d = 6{,}786$\rev{, but a direct power demonstration at that dimension is above our computational budget, since each replicate requires two dense $6{,}786 \times 6{,}786$ determinants per permutation, a cost that scales as $d^3 = O(p^6)$, and the $d \times d$ scatter matrices must be held in memory. This is why the series stops at $p = 40$.}

The second feature is that the Riemannian ANOVA and tangent-space MANOVA are indistinguishable at this design. Over the $44$ power configurations the two empirical powers differ by at most $0.045$, with a mean difference of $-0.002$. On the same replicates the two tests reach the same decision at $\alpha = 0.05$ in $96.6\%$ of the $8{,}000$ replicates with $\theta > 0$, and the rank correlation of their $p$-values is $0.94$. In every power configuration, and in every null configuration except $p = 25$ with $n_l \geq 205$, the matrix $W_E$ is singular and the comparison runs through the regularization step, so the agreement is between the regularized statistics. Section~\ref{sec:sim-basepoint} compares the two tests where the scatter matrices have full rank.

\begin{figure}[t]
\centering
\includegraphics[width=0.95\linewidth]{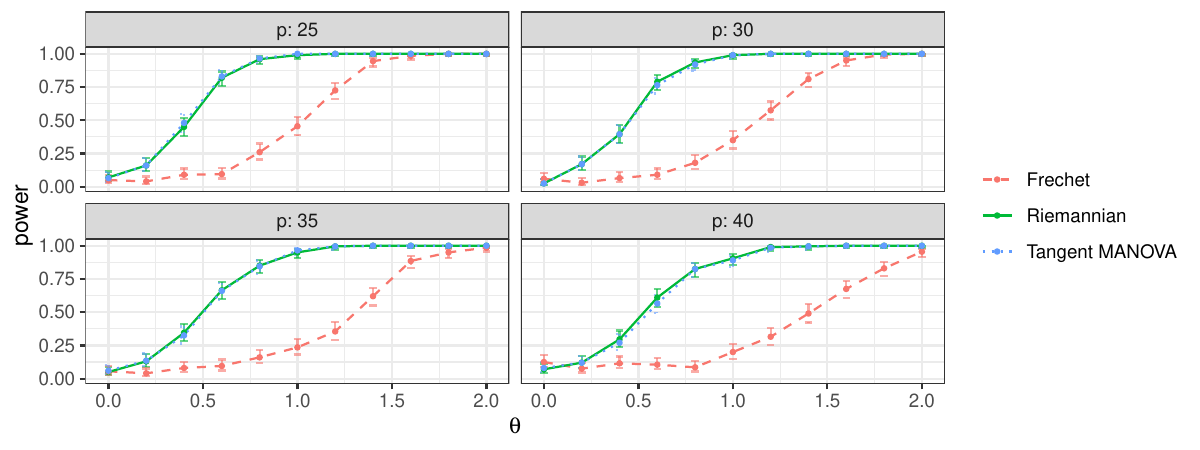}
\caption{Empirical power of the Riemannian ANOVA (solid), tangent-space MANOVA at the pooled Fr\'echet mean (dotted) and Fr\'echet ANOVA (dashed) versus location-shift size $\theta$, at $n_l = 100$ per group with $\Gamma$ fitted to connectome data, faceted by parcellation size $p \in \{25, 30, 35, 40\}$ ($d = 325, 465, 630, 820$). Error bars are Wilson $95\%$ confidence intervals over $R = 200$ replicates per point ($B = 49$). The two tangent-space tests coincide within Monte-Carlo error at every point. Fr\'echet ANOVA requires a shift between two and three times larger to reach the same power, and the shift it requires grows with $d$, consistent with the $1/d$ factor in its non-centrality (Proposition~\ref{prop:noncentrality_comparison}).}
\label{fig:sim-power}
\end{figure}

\subsection{Comparison with tangent-space MANOVA}\label{sec:sim-basepoint}

The first study compared the two tangent-space tests at dimensions where $d$ exceeds $n - g$ and both statistics pass through the regularization step. The second compares them where the scatter matrices have full rank. The two statistics differ through the base point. The proposed statistic forms each group's residuals in the tangent space at that group's own Fr\'echet mean. Tangent-space MANOVA forms every quantity in the tangent space at one point, so the group means it compares are Euclidean averages of projected vectors and the within-group scatter it pools is that of the projected clouds. These differences grow with the distances from the base point to the group means, measured against the within-group dispersion. The design therefore varies the dispersion, the number of groups, the balance of the group sizes and the placement of the group means, at the two dimensions $p = 5$ and $p = 10$ ($d = 15$ and $55$), with $n = 100$ in every cell.

\emph{Design.}\enspace The data are generated from the Riemannian normal law as in the first study, with $(\samFr, \widehat{\Gamma})$ refitted at $p = 5$ and $p = 10$ from the leading principal submatrices of the HCP-YA connectomes. The tangent covariance is $s^2\widehat{\Gamma}$ with $s \in \{1, 2, 4\}$. At $s = 1$ the root mean squared geodesic distance of an observation from its group mean is $1.5$ at $p = 5$ and $2.5$ at $p = 10$. These radii scale with $s$. Group $l$ is drawn from $\mathrm{RiemN}(C_l^\ast, \Gamma_l)$, where $\Gamma_l$ is the parallel transport of $s^2\widehat{\Gamma}$ along the geodesic from $\samFr$ to $C_l^\ast$, so the groups share one tangent covariance up to transport and differ in location only. There are $g \in \{2, 3\}$ groups, with sizes $50/50$ and $20/80$ for $g = 2$ and $34/33/33$ and $15/15/70$ for $g = 3$. Under the null hypothesis every $C_l^\ast$ equals $\samFr$. Under the alternative with $g = 2$, the first group, the smaller one in the unbalanced design, is displaced from $\samFr$ to Riemannian distance $\theta$ along the unit tangent direction of the first study. With $g = 3$ two placements are used. In the first, the three group means lie at distance $\theta$ from $\samFr$ along tangent directions $120^\circ$ apart. In the second, two groups remain at $\samFr$ and the third, the largest in the unbalanced design, is displaced to distance $\theta$. Four values of $\theta$ per $(p, s)$ were set from a pilot run to target powers of $\Lambda^\ast$ from $0.2$ to $0.9$ in the balanced two-group design; they range from $0.32$ to $0.87$ at $(p, s) = (5, 1)$ and from $1.49$ to $3.60$ at $(10, 4)$. The design has $24$ null cells and $144$ power cells, each of $R = 200$ replicates.

\emph{Tests.}\enspace Each replicate is analyzed with three permutation tests on the same data: $\Lambda^\ast$; tangent-space MANOVA at the pooled Fr\'echet mean $\hat{C}$, the statistic \eqref{eq:tangent-manova}; and tangent-space MANOVA at the identity matrix, the base point of \citet{varo}, at which the data are projected without first estimating their center. Fr\'echet ANOVA is not run. The null arm uses $B = 99$ relabelings, so that the attainable sizes are the nominal levels $0.01$, $0.05$ and $0.10$; the power arm uses $B = 49$, with attainable size $0.06$ at the nominal $0.05$ as in the first study. For $\Lambda^\ast$ the group Fr\'echet means are recomputed for every relabeling. The two tangent-space tests share one set of relabelings, applied to projections computed once, since both base points are invariant under relabeling. Every cell has $n - g \geq 97 > d$, so $W$, $T$ and their Euclidean counterparts have full rank and no null direction enters the regularization step of Appendix~\ref{sec:supp-computation}.

\emph{Type~I error.}\enspace Pooled over the $4{,}800$ replicates of the $24$ null cells, the rejection rates at the nominal levels $0.01$, $0.05$ and $0.10$ are $0.008$, $0.044$ and $0.098$ for $\Lambda^\ast$, $0.009$, $0.050$ and $0.098$ for tangent-space MANOVA at $\hat{C}$, and $0.008$, $0.048$ and $0.098$ at the identity. Of the $72$ cell-level rates per test, $70$, $66$ and $69$ lie within two Monte-Carlo standard errors of the attainable size. The largest departure is $2.9$ standard errors. The full table is in the Supplementary Material. Appendix~\ref{sec:supp-simulation} reports a rerun on fresh replicates of two cells whose rates exceeded the size at the $0.10$ level.

\emph{Power.}\enspace Figure~\ref{fig:sim-basepoint} plots the power of the three tests against $\theta$ at $p = 10$ in the four designs with one displaced group. In the two designs with the group means at the vertices of a triangle, all three tests have power above $0.99$ from the third value of $\theta$ on at both dimensions; those panels, and the full set at $p = 5$, are in Appendix~\ref{sec:supp-simulation}. Table~\ref{tab:sim-basepoint} summarizes the paired comparison over the $144$ power cells. At $p = 10$ the proposed test has more power than both tangent-space tests, on average over cells, at every dispersion scale. Averaged over the $24$ cells at each $(p, s)$, saturated cells included, its gain over tangent-space MANOVA at $\hat{C}$ is $0.034$ at $s = 1$, $0.058$ at $s = 2$ and $0.057$ at $s = 4$, with $95\%$ confidence intervals that exclude zero. Its gain over the identity-based test is the same to within $0.003$. At $p = 5$ the gains are $0.010$, $0.032$ and $0.001$. The gain at $p = 10$ is present in the balanced two-group design, $0.024$, $0.063$ and $0.095$ at $s = 1, 2, 4$, so it is not confined to the unbalanced and three-group designs. The largest gain in a single cell is $0.195$, power $0.655$ against $0.460$, at $p = 10$ and $s = 4$ with three balanced groups, one displaced to $\theta = 3.6$. Against each comparator taken separately, $44$ of the $144$ cells have a paired difference whose interval excludes zero; $43$ favor $\Lambda^\ast$ and one favors the comparator, by $0.040$ at $p = 5$ and $s = 1$. Under an exact McNemar test of the paired decisions with a Bonferroni correction over the $144$ cells, $10$ cells against the pooled mean and $11$ against the identity remain, all at $p = 10$ and all in favor of $\Lambda^\ast$. Size does not account for the gain. At the attainable size $0.06$ of the power arm, the null rejection rates at $p = 10$ are $0.053$ for $\Lambda^\ast$, $0.058$ for tangent-space MANOVA at $\hat{C}$ and $0.051$ at the identity. Diagnostics of the quantities through which the base point enters the two statistics are reported in Appendix~\ref{sec:supp-simulation}. They do not account for the gain, which we report as an empirical finding of these designs.

The two base points give tangent-space MANOVA the same power on average. Over the $144$ cells the identity-based test minus the pooled-mean test averages $0.000$ and ranges from $-0.060$ to $0.045$. The rank correlation of the two $p$-values is $0.97$. The decisions at the $0.05$ level differ on $4.2\%$ of the replicates in the power arm, and on up to $13.5\%$ in a single cell, so the base point moves the decision on a given sample without moving the power. The decisions of $\Lambda^\ast$ and of tangent-space MANOVA at $\hat{C}$ differ on $9.7\%$ of the replicates, and on up to $29.5\%$ in a cell.

\begin{table}[t]
\centering
\caption{Second simulation study: power of the Riemannian ANOVA at the attainable size $0.06$ and its gain over tangent-space MANOVA at the pooled Fr\'echet mean and at the identity ($B = 49$, $R = 200$). Each row is the mean over the $24$ power cells at that $(p, s)$, saturated cells included. The intervals are the mean paired difference plus or minus $1.96$ times its Monte-Carlo standard error, pooled over the $24$ cells from each cell's standard error of the mean difference of the two rejection indicators over its $R = 200$ replicates. The last three columns refer to the comparison with the pooled mean: the number of cells whose paired interval lies entirely above zero, the number entirely below zero, and the largest gain in a single cell.}
\label{tab:sim-basepoint}
\begin{tabular}{@{}cccccccc@{}}
\toprule
$p$ & $s$ & Power of $\Lambda^\ast$ & Gain over pooled mean & Gain over identity & Above $0$ & Below $0$ & Largest \\
\midrule
$5$ & $1$ & $0.631$ & $0.010$ $(0.002,\,0.017)$ & $0.007$ $(-0.000,\,0.015)$ & $3$ & $1$ & $0.075$ \\
$5$ & $2$ & $0.580$ & $0.032$ $(0.024,\,0.040)$ & $0.031$ $(0.023,\,0.039)$ & $9$ & $0$ & $0.090$ \\
$5$ & $4$ & $0.637$ & $0.001$ $(-0.008,\,0.011)$ & $0.002$ $(-0.007,\,0.012)$ & $1$ & $0$ & $0.075$ \\
$10$ & $1$ & $0.649$ & $0.034$ $(0.026,\,0.042)$ & $0.035$ $(0.027,\,0.043)$ & $9$ & $0$ & $0.125$ \\
$10$ & $2$ & $0.624$ & $0.058$ $(0.050,\,0.067)$ & $0.060$ $(0.051,\,0.069)$ & $11$ & $0$ & $0.165$ \\
$10$ & $4$ & $0.583$ & $0.057$ $(0.046,\,0.067)$ & $0.054$ $(0.043,\,0.064)$ & $10$ & $0$ & $0.195$ \\
\bottomrule
\end{tabular}
\end{table}

\begin{figure}[tp]
\centering
\includegraphics[width=\linewidth]{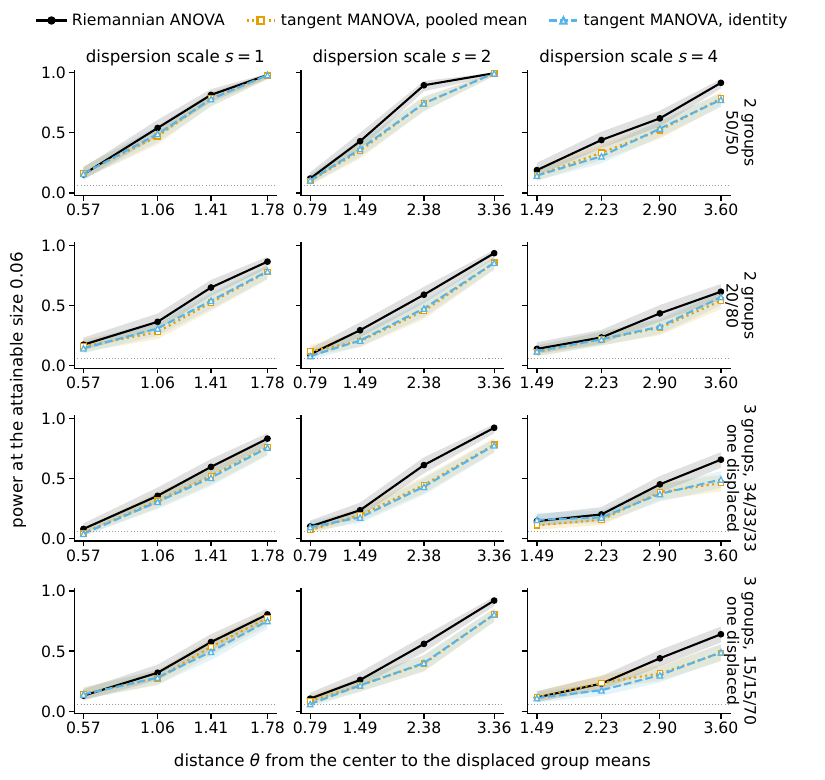}
\caption{Second simulation study at $p = 10$ ($d = 55$, $n = 100$): empirical power at the attainable size $0.06$ of the Riemannian ANOVA (solid, filled circles), tangent-space MANOVA at the pooled Fr\'echet mean (dotted, squares) and tangent-space MANOVA at the identity (dashed, triangles), against the Riemannian distance $\theta$ of the displaced group mean from the center, in the four designs with one displaced group (rows) and at the three dispersion scales $s$ (columns). Bands are Wilson $95\%$ confidence intervals over $R = 200$ replicates ($B = 49$); the horizontal dotted line is the attainable size. The values of $\theta$ differ by column. The two three-group designs with the group means at the vertices of a triangle, and the full set at $p = 5$, are in Appendix~\ref{sec:supp-simulation}.}
\label{fig:sim-basepoint}
\end{figure}

\section{Application: Distinguishing Connectome Samples}\label{sec:application}

We now apply the Riemannian ANOVA to resting-state functional connectivity data from the Human Connectome Project (HCP) and use it to \rev{illustrate} the power comparison of Section~\ref{sec:frechet-comparison}. We pursue two ends. First, \emph{specificity}: the test should not manufacture differences where none are expected. Second, and centrally, the theory predicts a \rev{regime in} which the Riemannian test detects an effect that distance-based Fr\'echet ANOVA misses. We exhibit two independent real-data instances of exactly that discordance, bracketed by a small-effect regime where the two tests agree on a null and a large-effect regime where they agree on a detection.

\subsection{Data and design}\label{sec:application-data}

We use \rev{several} HCP datasets. \rev{First, the} connectomes analyzed here, and the regularization that renders them positive definite, follow the construction of \citet{abbas2021}, to whom we refer for the details of preprocessing and estimation. The sex and phase-encoding comparisons use HCP-YA resting-state timeseries \citep{hcp}, parcellated into the $14$ functional networks of the Yeo atlas \citep{yeo}; for each subject the $14 \times 1190$ timeseries was converted to a $14 \times 14$ Pearson correlation matrix, projected to the nearest positive definite matrix by the algorithm of \citet{higham}. \rev{Second, the} site and cross-study comparisons use \rev{HCP-Aging} connectomes in the Schaefer 100-region cortical parcellation \citep{schaefer2018} augmented with the 16-region Tian S1 subcortical parcellation \citep{tian2020}, yielding $116 \times 116$ \rev{matrices. HCP-Aging} matrices were obtained pre-computed from a public repository \rev{\citep{tijhuis2022}. Third, HCP-YA} matrices were computed from the dense timeseries under the same parcellation.
We test five two-group comparisons. Two are negative controls in which no difference is expected: repeated resting-state scans of the same subjects (REST1 versus REST2, $n = 80$), and the two phase-encoding directions (left-to-right versus right-to-left, $n = 426$). The remaining three target genuine contrasts: \emph{sex} (HCP-YA, $n = 203$ male and $n = 223$ female), for which prior work provides strong evidence of a real effect \citep{weis2020, ritchie2018}; \emph{acquisition site} within HCP-Aging (University of Minnesota versus University of California, Los Angeles), examined both at a subsampled size ($n = 20$ per group) and at full sample ($n_{\text{UMinn}} = 205$, $n_{\text{UCLA}} = 148$); and a \emph{cross-study} contrast between HCP-YA and HCP-Aging ($n = 20$ per group).

Both Fr\'echet ANOVA and the Riemannian ANOVA were applied under \rev{AIRM}\rev{. Inference for the Riemannian ANOVA is by permutation as set out in Section~\ref{sec:permutation}}. The model diagnostics of Appendix~\ref{sec:supp-diagnostics} give a further, data-specific reason to prefer it here: they show a heavier-than-expected right tail in the tangent features, exactly the regime in which the asymptotic $\chi^2$ reference is least trustworthy, and permutation sidesteps that reliance entirely. \rev{We use $B = 100$ permutations, except for the phase-encoding and sex comparisons, which use $B = 20$; $\hat{p}$ is therefore resolved to $1/B$, and an observed statistic that exceeds every permuted value is reported as $p < 0.01$ at $B = 100$ and as $p < 0.05$ at $B = 20$. Fr\'echet ANOVA $p$-values are from its asymptotic $\chi^2_{1}$ reference \citep{muller}: the statistic is scalar, with $g - 1 = 1$ degree of freedom.} Computations use the \CRANpkg{riemtan} package \citep{riemtan}, which implements the affine-invariant logarithm, exponential, and vectorization maps on $\SPD_p$, together with the companion \CRANpkg{riemstats} package for the test statistics. Tangent-space MANOVA is not applied here. In the first study of Section~\ref{sec:simulation}, at the dimensions nearest the application, its power did not differ from that of the Riemannian ANOVA at any configuration, and at $d = 6{,}786$ it would be computed through the same regularization step as $\Lambda^*$.

\subsection{Results}\label{sec:application-results}

Table~\ref{tab:application} collects the results. The two negative controls behave as they should: neither test detects a difference between REST1 and REST2 or between phase-encoding directions, confirming that the Riemannian ANOVA does not produce spurious detections. The three substantive comparisons separate cleanly into three regimes.

In the sex comparison, the two tests disagree. Fr\'echet ANOVA does not detect a difference ($F = 0.964$, $p = 0.326$), while the Riemannian ANOVA \rev{does: the observed statistic exceeded all $B = 20$ permuted values, so we report $p < 0.05$, the resolution of that run; the two re-estimations of the same comparison in Appendix~\ref{sec:supp-estimation}, each with $B = 100$, return $p < 0.01$}. This discordance is biologically credible --- \rev{\citet{weis2020} classify sex from resting-state HCP connectivity with peak parcelwise accuracies of $75\%$ within sample and $66$--$73\%$ out of sample, and \citet{ritchie2018} report sex differences in resting-state connectome organization in $5{,}216$ UK Biobank participants} --- and it is robust to the connectivity estimator: re-estimating the correlations under oracle-approximating and empirical-Bayes shrinkage leaves the Riemannian detection \rev{in place} and Fr\'echet ANOVA non-significant (Table~\ref{tab:shrinkage} in Appendix~\ref{sec:supp-estimation}).

The full-sample \emph{site} comparison reproduces the same pattern in an independent cohort and at a far larger ambient dimension. With all $205$ UMinn and $148$ UCLA subjects, Fr\'echet ANOVA remains non-significant ($F = 1.491$, $p = 0.222$) while the Riemannian ANOVA detects a site effect ($p < 0.01$). At the subsampled size ($n = 20$ per group), by contrast, both tests return a null ($F = 1.359$, $p = 0.244$; Riemannian $p = 0.820$): \rev{the effect appears small, and at $n = 20$ neither test resolves it}. The full-sample site discordance does not persist when connectivity is re-estimated with empirical-Bayes shrinkage toward the grand mean, as reported in Appendix~\ref{sec:supp-estimation}.

Finally, the \emph{cross-study} comparison is a large effect that both tests detect easily: Fr\'echet $F = 58.668$, $p < 0.001$, and Riemannian $p < 0.01$. The two studies differ in scanner, protocol, and the age range of their participants, and the resulting connectivity structures are plainly distinguishable.

\begin{table}[t]
\centering
\caption{Two-group comparisons of connectome populations under the affine-invariant metric. \protect\rev{Riemannian $p$-values are by permutation with $B$ relabelings and are resolved to $1/B$; Fr\'echet $p$-values are from the asymptotic $\chi^2_1$ reference.} Rows are ordered by the regime they illustrate. Sample sizes are per group unless noted ($n_{\text{UMinn}}/n_{\text{UCLA}}$ for the full-sample site comparison).}
\label{tab:application}
\begin{tabular}{llrrrr}
\toprule
Comparison & $n$ & Fr\'echet $F$ & Fr\'echet $p$ & \rev{$B$} & Riemannian $p$ \\
\midrule
REST1 vs.\ REST2          & 80       & 0.260  & 0.610          & \rev{100} & 0.990 \\
RL vs.\ LR                & 426      & 0.128  & 0.721          & \rev{20}  & 0.950 \\
Site, subsampled          & 20       & 1.359  & 0.244          & \rev{100} & 0.820 \\
Sex                       & 203/223  & 0.964  & 0.326          & \rev{20}  & \rev{$<0.05$} \\
Site, full sample         & 205/148  & 1.491  & 0.222          & \rev{100} & $<0.01$ \\
Cross-study (YA vs.\ Aging)& 20      & 58.668 & $<0.001$       & \rev{100} & $<0.01$ \\
\bottomrule
\end{tabular}
\end{table}

The six comparisons trace out \rev{the three regimes} that the non-centrality comparison of Proposition~\ref{prop:noncentrality_comparison} predicts. There, a location shift of Riemannian size $\theta$ contributes $\delta_{\text{Riemannian}} = O(\theta^2)$ to the Riemannian test's non-centrality but only $\delta_{\text{Fr\'echet}} = O(\theta^4)$ to that of Fr\'echet ANOVA. For small $\theta$ both non-centralities are below what the sample size can resolve, and both tests return a null --- the subsampled site comparison, together with the negative controls, occupies this regime. For large $\theta$ both non-centralities are large and both tests detect, as in the cross-study comparison. Between these lies a band of moderate shifts in which $\delta_{\text{Riemannian}}$ is already large while $\delta_{\text{Fr\'echet}}$, smaller by a factor of order $\theta^2/(\sigma^2 d)$, is \rev{not.} The sex comparison and the full-sample site comparison are two independent real-data realizations of this moderate-effect band, in different cohorts and at ambient tangent dimensions of $d = 105$ and $d = 6{,}786$ respectively. That the discordance appears twice, in unrelated contrasts, provides direct empirical support for the theoretical advantage established in Section~\ref{sec:frechet-comparison}.

\section{Discussion}\label{sec:discussion}

In this paper we studied the problem of comparing samples that take values in the Riemannian manifold $\SPD_p$. This manifold is in particular a metric space, so at least two methods appear as natural options, distance components (DISCO) analysis \citep{rizzo2010} and Fr\'echet ANOVA \citep{muller}. Appendix~\ref{sec:supp-negative-type} shows that the first is not applicable, at least not directly, because of technical properties of the metric. As for the second, we showed that it lacks the power needed to reject the null hypothesis in two practical situations in which independent evidence indicates that it should be rejected. These are the comparison of male and female participants in HCP-YA and the comparison of connectomes acquired at two different sites within HCP-Aging.

Both of these methods are generic in their formulation, in the sense that they are valid for a wide range of metric spaces. This generality suggests that additional statistical power can be obtained by using the particular features of the manifold. A third option is to project the samples to a common tangent space, vectorize, and apply Euclidean MANOVA, an option that we call tangent-space MANOVA. This option does not reflect the geometry of the sample faithfully, since the group mean of the tangent projections is not the tangent projection of the group Fr\'echet mean.

For that reason we introduced a new method, the Riemannian ANOVA. It uses a sum-of-squares decomposition formed in the tangent spaces at the group Fr\'echet means to obtain a test statistic that generalizes the Euclidean Wilks' Lambda. The method is more faithful to the geometry of the problem. In simulations it has a modest but detectable power advantage over tangent-space MANOVA. Its advantage over Fr\'echet ANOVA is more pronounced. It appears in simulations and also in practice, where the Riemannian ANOVA rejects in the two real scenarios mentioned above. We also built a conceptual framework that explains this difference in power. We described the asymptotic behavior of the proposed method and of Fr\'echet ANOVA under alternatives close to the null hypothesis and under a parametric distribution that we introduced, whose relevance to real connectomes is discussed in Appendix~\ref{sec:supp-diagnostics}. This description brings out differences that may explain the empirical observations.

These findings have implications for the discussion of how effective harmonization methods are for connectomes acquired in multi-site studies. That discussion relies mostly on the ability of machine learning classifiers to distinguish samples from different sites before and after harmonization. This approach provides valid evidence. A formal statistical test such as the one proposed here would complement it. More generally, this paper demonstrates a principle that we suspect holds more widely. For samples taking values in spaces with geometric structure, some generality in the formulation of a test can be given up in exchange for additional statistical power.

This study has several limitations. The most notable is perhaps that all the data in our experiments come from studies that are part of the HCP. Future studies should consider more heterogeneous data. The power comparison in simulations also required several choices that may limit how far it generalizes to other settings, such as the specific form of the alternative and the parametric distribution used. A more general study of the behavior of the tests under a wider range of alternatives and under weaker distributional requirements would also be worthwhile. Finally, the implementation of the proposed test and of the other tests considered in this paper requires choices related to numerical stability, which are described in Appendix~\ref{sec:supp-computation}.

\begin{appendix}
\section{Uniqueness of the Fr\'echet mean}\label{sec:supp-uniqueness}

The test statistic and the null hypothesis are defined in terms of population and sample Fr\'echet means, so we record that these are well defined under the Riemannian normal model. Write $F(M) = \tfrac{1}{2}\,\Expe\, d^2(M, \conn)$ for the Fr\'echet functional of $\conn \sim p$, where $d$ is the geodesic distance of the affine-invariant metric; a Fr\'echet mean is a global minimizer of $F$.

\begin{proposition}[Uniqueness of the Fr\'echet mean]\label{prop:frechet-uniqueness}
Let $\conn$ follow the Riemannian normal model with parameters $\conn^\ast \in \SPD_p$ and $\Gamma \succ 0$. Then $\conn^\ast$ is the unique global Fr\'echet mean of $\conn$.
\end{proposition}

The argument rests on the nonpositive curvature of the manifold. We use the following quadratic comparison inequality, valid for any minimizing geodesic $\gamma:[0,1]\to\SPD_p$ and any $C\in\SPD_p$:
\begin{equation}\label{eq:qc}
d^2(\gamma(t), C) \;\leq\; (1-t)\, d^2(\gamma(0), C) + t\, d^2(\gamma(1), C) - t(1-t)\, d^2(\gamma(0), \gamma(1)), \quad t\in[0,1].
\end{equation}
Inequality~\eqref{eq:qc} is the CAT($0$) comparison condition. The manifold $\SPD_p$ with the affine-invariant metric is complete, simply connected, and of nonpositive sectional curvature, hence a Hadamard manifold; by the Rauch comparison theorem the exponential map is distance-non-decreasing relative to its Euclidean model, and by the Cartan--Hadamard theorem it is a global diffeomorphism, so the comparison condition holds globally \citep{docarmo}. Applying it to the triangle with vertices $\gamma(0),\gamma(1),C$ at the point $\gamma(t)$ and translating the Euclidean comparison identity back into geodesic distances yields~\eqref{eq:qc}.

\begin{proof}[Proof of Proposition~\ref{prop:frechet-uniqueness}]
Pushing the law of $\conn$ forward through $\log_{\conn^\ast}$ and $\operatorname{Vec}_{\conn^\ast}$ gives a centered Gaussian on $\Reals^d$ with covariance $\Gamma$; since $\operatorname{Vec}_{\conn^\ast}$ is an isometry, $\Expe\, d^2(\conn^\ast,\conn) = \operatorname{tr}(\Gamma) < \infty$, so $F(\conn^\ast)<\infty$, and by the triangle inequality $F(M)<\infty$ for every $M$. The first-variation formula $\operatorname{grad}F(M) = -\Expe[\log_M(\conn)]$ together with $\Expe[\log_{\conn^\ast}(\conn)] = 0$ shows that $\conn^\ast$ is a critical point of $F$. Taking expectations in~\eqref{eq:qc} (the convexity defect $t(1-t)\,d^2(\gamma(0),\gamma(1))$ is uniform in $C$) gives, for every non-constant geodesic $\gamma$ and $t\in(0,1)$,
\begin{equation}\label{eq:strict-convexity}
F(\gamma(t)) \;\leq\; (1-t)\,F(\gamma(0)) + t\,F(\gamma(1)) - \tfrac{1}{2}\,t(1-t)\,d^2(\gamma(0),\gamma(1)),
\end{equation}
so $F$ is strictly convex along geodesics. A strictly convex, finite, continuous function on a Hadamard manifold has at most one critical point, which is then its unique global minimizer; since $\conn^\ast$ is such a critical point, it is the unique global Fr\'echet mean.
\end{proof}

\section{Supporting lemmas and proofs for the asymptotic theory}\label{sec:supp-proofs}

This section records the estimates on which the asymptotic theory of Section~\ref{sec:asymptotics} depends, together with the proofs of the results stated there. Throughout, the population and sample Fr\'echet means are well defined by Proposition~\ref{prop:frechet-uniqueness}.

\subsection*{Supporting lemmas for the asymptotic theory}

The proof of Theorem~\ref{thm:asymptotic} requires control of how the Riemannian logarithm changes as its base point moves along a geodesic. The following estimates supply that control.

\begin{theorem}[Jacobi Field Quotient Identity]\label{thm:jacobi-field}
Let $M$ be a Riemannian manifold and $\gamma : [0,1] \to M$ be the geodesic $\gamma(s) = \exp_C(s\log_C(A))$ connecting $C$ to $A$. Consider the variation
$\sigma(s,t) = \exp_{\gamma(s)}(t\log(B))$.
Let $J(s) = \frac{\partial \sigma(s,t)}{\partial t}\big|_{t=0}$ be the associated Jacobi field. Then for small $\varepsilon > 0$,
\begin{equation}\frac{\log_{\gamma(\varepsilon)}(B) - \Pi_{C \to \gamma(\varepsilon)}\left[\log_C(B)\right]}{\varepsilon} = \frac{D}{ds} J\bigg|_{s=\varepsilon} + O(\varepsilon),\end{equation}
where $\Pi_{C \to \gamma(\varepsilon)}$ denotes parallel transport along $\gamma$ from $C$ to $\gamma(\varepsilon)$.
\end{theorem}

\begin{proof}
We establish the result through an analysis of the Jacobi field evolution and its covariant derivatives.
Let $\{E_i(s)\}_{i=1}^n$ be a parallel orthonormal frame along $\gamma$, constructed by parallel transporting an orthonormal basis from $T_C M$ along the geodesic. Since the frame is parallel, $\frac{D}{ds} E_i = 0$ for all $i$ and $s \in [0,1]$.
Express the Jacobi field in this frame, $J(s) = \sum_{i=1}^n J^i(s) E_i(s)$, where $J^i(s) = \langle J(s), E_i(s) \rangle$.
Since the frame is parallel,
\begin{equation}\frac{D}{ds} J = \sum_{i=1}^n \frac{dJ^i}{ds} E_i(s), \quad \frac{D^2}{ds^2} J = \sum_{i=1}^n \frac{d^2J^i}{ds^2} E_i(s).\end{equation}
The Jacobi equation $\frac{D^2}{ds^2} J + R(J, \dot{\gamma})\dot{\gamma} = 0$ becomes, in components,
\begin{equation}\frac{d^2J^i}{ds^2} + \sum_{j=1}^n R_{ij}(s) J^j(s) = 0, \quad i = 1, \ldots, n,\end{equation}
where $R_{ij} = \langle R(E_j, \dot{\gamma})\dot{\gamma}, E_i \rangle$. Taylor-expanding each $J^i$ around $s = 0$ and using the Jacobi equation at $s=0$ to substitute for $\frac{d^2 J^i}{ds^2}|_{s=0}$,
\begin{equation}J^i(\varepsilon) = J^i(0) + \varepsilon\frac{dJ^i}{ds}\bigg|_{s=0} - \frac{\varepsilon^2}{2}\sum_{j=1}^n R_{ij}(0) J^j(0) + O(\varepsilon^3).\end{equation}
The parallel transport of $J(0) \in T_C M$ to $T_{\gamma(\varepsilon)} M$ is $\Pi_{C \to \gamma(\varepsilon)}[J(0)] = \sum_{i=1}^n J^i(0) E_i(\varepsilon)$, so
\begin{equation}\frac{J(\varepsilon) - \Pi_{C \to \gamma(\varepsilon)}[J(0)]}{\varepsilon} = \sum_{i=1}^n \frac{dJ^i}{ds}\bigg|_{s=0} E_i(\varepsilon) - \frac{\varepsilon}{2}\sum_{i,j=1}^n R_{ij}(0) J^j(0) E_i(\varepsilon) + O(\varepsilon^2).\end{equation}
Tracking $\frac{D}{ds}J$ along $\gamma$ via the Jacobi equation gives
\begin{equation}\frac{D}{ds} J\big|_{s=\varepsilon} = \Pi_{C \to \gamma(\varepsilon)}\Big[\tfrac{D}{ds} J\big|_{s=0}\Big] - \varepsilon \Pi_{C \to \gamma(\varepsilon)}[R(J(0), \dot{\gamma}(0))\dot{\gamma}(0)] + O(\varepsilon^2),\end{equation}
and substituting, the curvature term is $O(\varepsilon)$, so
\begin{equation}\frac{J(\varepsilon) - \Pi_{C \to \gamma(\varepsilon)}[J(0)]}{\varepsilon} = \frac{D}{ds} J\big|_{s=\varepsilon} + O(\varepsilon).\end{equation}
Finally, since $\gamma(\varepsilon) = \exp_C(\varepsilon \log_C(A))$, $J(\varepsilon) = \log_{\gamma(\varepsilon)}(B)$ and $J(0) = \log_C(B)$, the claim follows.
\end{proof}

\begin{corollary}[Logarithm Difference Formula]\label{cor:log-difference}
Let $C, A, B \in M$ and consider the geodesic $\gamma(s) = \exp_C(s\log_C(A))$. For small $\varepsilon > 0$,
\begin{equation}\frac{\log_{\gamma(\varepsilon)}(B) - \Pi_{C \to \gamma(\varepsilon)}[\log_C(B) - \log_C(A)] - \log_{\gamma(\varepsilon)}(A)}{\varepsilon} = \frac{D}{ds} \tilde{J}\bigg|_{s=\varepsilon} + O(\varepsilon),\end{equation}
where $\tilde{J}(s)$ is the Jacobi field along $\gamma$ associated to the variation
\[\sigma(s,t) = \exp_{\gamma(s)}(t[\log_{\gamma(s)}(B) - \log_{\gamma(s)}(A)]).\]
\end{corollary}

\begin{proof}
This variation satisfies $\sigma(s,0) = \gamma(s)$, $\sigma(0,t) = \exp_C(t[\log_C(B) - \log_C(A)])$, and $\sigma(1,t) = \exp_A(t \log_A(B))$. The associated Jacobi field is $\tilde{J}(s) = \log_{\gamma(s)}(B) - \log_{\gamma(s)}(A)$. Decomposing $\log_{\gamma(\varepsilon)}(B) = [\log_{\gamma(\varepsilon)}(B) - \log_{\gamma(\varepsilon)}(A)] + \log_{\gamma(\varepsilon)}(A)$, the bracketed term is $\tilde{J}(\varepsilon)$, which by Theorem~\ref{thm:jacobi-field} equals $\Pi_{C \to \gamma(\varepsilon)}[\tilde{J}(0)] + \varepsilon \frac{D}{ds}\tilde{J}|_{s=\varepsilon} + O(\varepsilon^2)$. Since $\tilde{J}(0) = \log_C(B) - \log_C(A)$, rearranging and dividing by $\varepsilon$ gives the result.
\end{proof}

\begin{proposition}[AIRM Jacobi Field Velocity]\label{prop:airm-jacobi-velocity}
Let $(\mathcal{M}, g^{\mathrm{AIRM}})$ be the manifold of $n \times n$ symmetric positive definite matrices equipped with the AIRM. Let $C, A, B \in \mathcal{M}$, and let $\gamma(s) = C^{1/2}(C^{-1/2}AC^{-1/2})^s C^{1/2}$ be the geodesic from $C$ to $A$. Define the Jacobi field $J(s)$ along $\gamma$ arising from the variation $\sigma(s,t) = \exp_{\gamma(s)}(t[\log_{\gamma(s)}(B) - \log_{\gamma(s)}(A)])$. Then
\begin{equation}\frac{D}{ds}J \bigg|_{s = 0} = -\frac{1}{2}[\log_C(A), \log_C(B)] - \frac{1}{2}\{\log_C(A), \log_C(B)\}_C + \log_C(A),\end{equation}
where $[X,Y] = XY - YX$ is the matrix commutator, $\{X,Y\}_C = C^{-1}XY + YX^TC^{-1}$ is the metric-symmetrized product at $C$, and $\log_P(Q) = P^{1/2}\log(P^{-1/2}QP^{-1/2})P^{1/2}$.
\end{proposition}

\begin{proof}
For the AIRM at a point $C$, the Levi-Civita connection is $\nabla_X Y = DY[X] - \frac{1}{2}(C^{-1}XY + YX^TC^{-1})$, where $DY[X]$ is the Euclidean directional derivative. Applying this with $\dot{\gamma}(0) = \log_C(A)$ and the constant extension of $\log_C(B)$,
\begin{equation}\frac{D}{ds}\bigg|_{s=0}\log_{\gamma(s)}(B) = -\frac{1}{2}(C^{-1}\log_C(A)\log_C(B) + \log_C(B)\log_C(A)^TC^{-1}).\end{equation}
Decomposing $\log_C(A)\log_C(B) = \frac{1}{2}[\log_C(A), \log_C(B)] + \frac{1}{2}\{\log_C(A), \log_C(B)\}$ and using $\frac{D}{ds}|_{s=0}\log_{\gamma(s)}(A) = -\log_C(A)$, the result follows by gathering terms.
\end{proof}

\begin{remark}
The formula in Proposition~\ref{prop:airm-jacobi-velocity} is explicit: all terms are computable by matrix logarithms (via eigendecomposition), multiplications, and additions, with no differential-equation solving required, making the curvature correction in the asymptotic expansion both efficient and numerically stable to evaluate.
\end{remark}

\subsection*{Proof of Theorem~\ref{thm:asymptotic}}

Recall that $\Lambda^* = |W|/|T|$ with $T = \sum_{l,j} w_{lj}w_{lj}^T$ and $W = \sum_{l,j} u_{lj}u_{lj}^T$. We apply the central limit theorem for Fr\'echet means \citep{clt4}. Define the chart $\phi: \mathrm{SPD}(p) \to \mathbb{R}^d$ by $\phi(C) = \operatorname{Vec}_{C^*}(\log_{C^*}(C))$. Writing $\rho_\phi^2(u, \theta) = \rho^2(\exp_{C^*}(u), \exp_{C^*}(\theta))$ for the squared geodesic distance in the chart and $\Psi(u;\theta) = \operatorname{grad}_\theta \rho_\phi^2(u,\theta)$, one has $\Psi(u;\theta) = -2G_{C^*}(u-\theta) + O(\|\theta\|^2)$, where $G_{C^*}$ is the matrix of the metric tensor at $C^*$ in vectorized coordinates. Under $H_0$, with $u \sim N(0,\Gamma)$,
\begin{equation}\Sigma = \operatorname{Cov}(\Psi(u;0)) = 4 G_{C^*}\Gamma G_{C^*}^T, \qquad \Lambda_{\text{H}} = E\Big[\tfrac{\partial}{\partial \theta_j}\Psi_i(u;\theta)\big|_{\theta=0}\Big] = 2G_{C^*}.\end{equation}
By Theorem 2.3 of \citet{clt4}, $\sqrt{n}(\phi(\hat C) - \phi(C^*)) \xrightarrow{d} N(0, \Lambda_{\text{H}}^{-1}\Sigma \Lambda_{\text{H}}^{-T}) = N(0,\Gamma)$; the metric terms cancel, giving
\begin{equation}\sqrt{n}\,\operatorname{Vec}_{C^*}(\log_{C^*}(\hat C)) \xrightarrow{d} N(0,\Gamma).\end{equation}
For the between-group deviation vectors $v_l = \operatorname{Vec}_{\hat C_l}(\log_{\hat C_l}(\hat C))$, Corollary~\ref{cor:log-difference} with $A = \hat C_l$, $B = \hat C$, $C = C^*$, together with $d_R(\hat C_l, C^*) = O_p(n_l^{-1/2})$ and $d_R(\hat C, C^*) = O_p(n^{-1/2})$, yields $\log_{\hat C_l}(\hat C) = \Pi_{C^*\to\hat C_l}[\log_{C^*}(\hat C) - \log_{C^*}(\hat C_l)] + O_p(n^{-1/2})$, whence $v_l = O_p(n^{-1/2})$. For the observation vectors $u_{lj} = \operatorname{Vec}_{\hat C_l}(\log_{\hat C_l}(C_{lj}))$, the same corollary gives $u_{lj} = \operatorname{Vec}_{C^*}(\log_{C^*}(C_{lj})) + O_p(n_l^{-1/2})$. Writing $u_{lj} = z_{lj} + \varepsilon_{lj}$ with $z_{lj}\sim N(0,\Gamma)$ exactly and $\varepsilon_{lj} = O_p(n_l^{-1/2})$,
\begin{equation}W = \sum_{l,j} u_{lj}u_{lj}^T = Z + E + R,\end{equation}
where $Z \sim \mathcal{W}_d(n-g, \Gamma)$ (the $g$ degrees of freedom lost to the centering constraints $\sum_j u_{lj}=0$), $E = O_p(n_l^{-1/2})$ collects the cross terms, and $R = O_p(n_l^{-1/4})$ is the quadratic error. The between-group term $B = \sum_l n_l v_l v_l^T = O_p(1)$. The remainder of the argument proceeds exactly as in Euclidean MANOVA \citep{anderson}, giving $-n\log\Lambda^* \xrightarrow{d} \chi^2_{d(g-1)}$. \qed

\subsection*{Non-centrality under a fixed alternative}

Propositions~\ref{prop:frechet-noncentrality}--\ref{prop:noncentrality_comparison} are stated in terms of a non-centrality, which we define here. Let $S_n$ denote either test statistic, $-n\log\Lambda^*$ or $T_n$. Fix limiting group proportions $\pi_l$ and population Fr\'echet means $C_1^*, \ldots, C_g^*$ that are not all equal. Under such a fixed alternative $S_n$ has no limiting distribution. It grows at rate $n$, and
\begin{equation}\label{eq:supp-tau}
\frac{S_n}{n} \convPr \tau
\end{equation}
for a constant $\tau > 0$ that depends on the alternative. We call $\delta = n\tau$ the \emph{non-centrality} of $S_n$ at that alternative. It is the deterministic part of the growth of the statistic. It is not the parameter of a non-central chi-square distribution, since no such limit holds at a fixed alternative. Suppose now that the alternative depends on a scalar $\theta \geq 0$ that measures its distance from the null hypothesis, with all means equal at $\theta = 0$, and that $\tau(\theta) = c\,\theta^k + o(\theta^k)$ as $\theta \to 0$ for a constant $c > 0$ and an integer $k \geq 1$. We call $n c\,\theta^k$ the \emph{leading-order non-centrality}. The proofs below compute the limit in \eqref{eq:supp-tau} for the two statistics and its leading term, with $k = 2$ for $-n\log\Lambda^*$ and $k = 4$ for $T_n$.

\subsection*{Proof of Proposition~\ref{prop:frechet-noncentrality}}

Suppose $C_{lj} \sim p(C_l^*)$ is Riemannian normal on $\mathrm{SPD}(p)$ with mean $C_l^*$ and tangent covariance $\Gamma$. Under equal covariances the population Fr\'echet variance for each group is $V_l = \mathbb{E}[d^2(C_{lj}, C_l^*)] = \operatorname{tr}(\Gamma)$, since $\log_{C_l^*}(C_{lj}) \sim N(0,\Gamma)$. The overall Fr\'echet variance decomposes (as in \eqref{eq:anova-decomp}) as $V_p = \sum_l \lambda_l(V_l + d^2(C_l^*,C^*)) = \operatorname{tr}(\Gamma) + \sum_l \lambda_l d^2(C_l^*,C^*)$. Hence the scale component vanishes, $U = \sum_{j<l}\frac{\lambda_j\lambda_l}{\sigma_j^2\sigma_l^2}(V_j-V_l)^2 = 0$, while the location component is $F = V_p - \sum_l \lambda_l V_l = \sum_l \lambda_l d^2(C_l^*,C^*)$. Since $\log_{C_l^*}(C_{lj}) \sim N(0,\Gamma)$, the squared distance $d^2(C_{lj}, C_l^*)$ is a Gaussian quadratic form with mean $\operatorname{tr}(\Gamma)$ and variance $2\operatorname{tr}(\Gamma^2)$, so $\hat{\sigma}_l^2 \to 2\operatorname{tr}(\Gamma^2)$ for every $l$. The scale term of $T_n$ therefore keeps its null $\chi^2_{g-1}$ limit to leading order, and the location term $nF_n^2/\sum_l \lambda_l^2\hat{\sigma}_l^2$ exceeds its null value by $nF^2/\{2\operatorname{tr}(\Gamma^2)\sum_l\lambda_l^2\} = \delta_{\text{Fr\'echet}}$ to leading order. Along $C_l^* = \exp_{C^*}(n^{-1/4}H_l)$ this shift converges to $(\sum_l\lambda_l\|\operatorname{Vec}_{C^*}(H_l)\|^2)^2/\{2\operatorname{tr}(\Gamma^2)\sum_l\lambda_l^2\}$. Along alternatives of order $n^{-1/2}$ it tends to zero, in agreement with Theorem 3 of \citet{muller}. \qed

\subsection*{Proof of Proposition~\ref{prop:noncentrality}}

The between-group scatter is $B = \sum_l n_l \hat v_l \hat v_l^T$. Since $\hat v_l \xrightarrow{p} v_l^* \neq 0$ for at least one $l$ under $H_1$, $B/n \xrightarrow{p} \sum_l \pi_l v_l^*(v_l^*)^T = B^*$, positive semidefinite of rank $\geq 1$, so $B = O_p(n)$. The within-group scatter $W = \sum_{l,i}(v_{li}-\hat v_l)(v_{li}-\hat v_l)^T$ satisfies $W/n \xrightarrow{p} W^* = \sum_l \pi_l C_l$, with $C_l$ the population within-group tangent covariance, so $W = O_p(n)$. Writing $\tilde W = W/n$, $\tilde B = B/n$, $\Lambda^* = |\tilde W|/|\tilde W + \tilde B| \xrightarrow{p} |W^*|/|W^*+B^*| = \lambda^* < 1$ strictly. Expanding $-n\log\Lambda^* = -n\log(|\tilde W|/|\tilde W + \tilde B|)$ via $\log|I+A| \approx \operatorname{tr}(A) - \frac12\operatorname{tr}(A^2)$ for $A = \tilde W^{-1}\tilde B$ gives the leading term $n\operatorname{tr}(W^{*-1}B^*) = n\sum_l \pi_l (v_l^*)^T(W^*)^{-1}v_l^*$, which is the stated $\delta$. Since $\delta_n = O(n)$, for any fixed critical value $c_\alpha$ the power $P(-n\log\Lambda^* > c_\alpha \mid H_1) \to 1$. \qed

\section{Negative type of the affine-invariant metric}\label{sec:supp-negative-type}

\begin{revblock}
\revnote{Editorial note for the marked copy: this section reproduces the definitions, theorem and proof that formed Section 9 of the 28 July draft, changed only in the two sentences that located it within the paper; the main text now refers to it only from the Introduction. The empirical typicality study of negative-type violations --- Section S2 of the 28 July supplement --- has been removed from the paper and from this supplement.}

The test developed in this paper is built on the Fr\'echet mean and a Wilks-type determinant ratio; its validity rests on the Riemannian normal model of Section~\ref{sec:model}, not on any metric-embedding property. This section records an independent structural fact about the affine-invariant metric, referred to in Section~\ref{sec:discussion}, that we find of interest in its own right and that bears on a different family of methods. It is not a prerequisite for anything in the main text.

Distance covariance, introduced by \citet{szekely2007} for Euclidean spaces and extended to general metric spaces by \citet{lyons2013}, provides a nonparametric measure of dependence; its extension requires the underlying space to be of \emph{negative type}, and strong negative type is both necessary and sufficient for distance covariance to characterize independence. The DISCO extension of ANOVA \citep{rizzo2010} requires the same condition. We show that $(\SPD_p, d_{\mathrm{AIRM}})$ is not of negative type for $p \geq 3$, so the assumptions underlying both DISCO and metric-space distance covariance fail in the connectomics setting --- a further, if secondary, motivation for the manifold-based approach taken here.

\begin{definition}[Negative Type]\label{def:negative-type}
A metric space $(X, d)$ has \emph{negative type} if for all $n \geq 1$, all $x_1, \ldots, x_n \in X$, and all $c_1, \ldots, c_n \in \mathbb{R}$ with $\sum_{i=1}^n c_i = 0$,
\[
\sum_{i,j=1}^n c_i\, c_j\, d(x_i, x_j) \leq 0.
\]
Equivalently, by Schoenberg's theorem \citep{schoenberg1938}, $(X, d)$ has negative type if and only if there exists a Hilbert space $\mathcal{H}$ and a map $\varphi: X \to \mathcal{H}$ such that $d(x,y) = \|\varphi(x) - \varphi(y)\|^2$ for all $x, y \in X$.
\end{definition}

\begin{definition}[Conditionally Negative Definite Function]\label{def:cnd}
A function $\psi: G \to \mathbb{R}$ on a group $G$ is \emph{conditionally negative definite} (CND) if $\psi(e) = 0$ and for all $n \geq 1$, all $g_1, \ldots, g_n \in G$, and all $c_1, \ldots, c_n \in \mathbb{R}$ with $\sum_{i=1}^n c_i = 0$,
\[
\sum_{i,j=1}^n c_i\, c_j\, \psi(g_i^{-1} g_j) \leq 0.
\]
\end{definition}

\begin{lemma}[Squared Hilbert Distance is CND]\label{lem:squared-hilbert-cnd}
For all $n \geq 1$, all $x_1, \ldots, x_n$ in a real Hilbert space $\mathcal{H}$, and all $c_1, \ldots, c_n \in \mathbb{R}$ with $\sum_i c_i = 0$, $\sum_{i,j} c_i\, c_j\, \|x_i - x_j\|^2 \leq 0$. See \citet[Example~C.2.2]{bekka2008}.
\end{lemma}

We recall two classical results from the theory of Property~(T).

\begin{theorem}[Kazhdan]\label{thm:kazhdan}
The group $\mathrm{SL}(n,\mathbb{R})$ has Kazhdan's Property~(T) for all $n \geq 3$ \citep[Theorem~1.4.15]{bekka2008}.
\end{theorem}

\begin{theorem}[Delorme--Guichardet]\label{thm:delorme-guichardet}
A locally compact, $\sigma$-compact group $G$ has Kazhdan's Property~(T) if and only if every conditionally negative definite function $\psi: G \to \mathbb{R}$ is bounded \citep[Theorem~2.12.4]{bekka2008}.
\end{theorem}

\begin{theorem}[$\SPD_p$ is Not of Negative Type for $p \geq 3$]\label{thm:spd-not-negative-type}
For $p \geq 3$, the metric space $(\SPD_p, d_{\mathrm{AIRM}})$ is not of negative type.
\end{theorem}

\begin{proof}
Suppose for contradiction that $(\SPD_p, d_{\mathrm{AIRM}})$ has negative type for some $p \geq 3$. Let $G = \mathrm{GL}(p,\mathbb{R})$ act on $\SPD_p$ by congruence, $g \cdot P = gPg^T$. This action is transitive, with stabilizer $K = O(p)$ of the identity $I_p$, so that $\SPD_p \cong G/K$, and the AIRM distance is $G$-invariant. Define $\psi: G \to \mathbb{R}$ by $\psi(g) := d_{\mathrm{AIRM}}(gg^T, I_p) = \|\log(gg^T)\|_F$. By the negative type hypothesis and Schoenberg's theorem there is a Hilbert space $\mathcal{H}$ and an embedding $\varphi$ with $d_{\mathrm{AIRM}}(P, Q) = \|\varphi(P) - \varphi(Q)\|^2$, so $\psi(g) = \|\varphi(g \cdot I_p) - \varphi(I_p)\|^2$. By $G$-invariance, for any $g_1, \ldots, g_n \in G$ and $c_1,\ldots,c_n$ with $\sum_i c_i = 0$,
\[
\sum_{i,j} c_i c_j \psi(g_i^{-1}g_j) = \sum_{i,j} c_i c_j \|\varphi(g_j \cdot I_p) - \varphi(g_i \cdot I_p)\|^2 \leq 0
\]
by Lemma~\ref{lem:squared-hilbert-cnd}, so $\psi$ is CND, and its restriction to $\mathrm{SL}(p,\mathbb{R})$ is too. However, $\psi$ is unbounded on $\mathrm{SL}(p,\mathbb{R})$: the family $g(t) = \mathrm{diag}(e^t, e^{-t/(p-1)}, \ldots, e^{-t/(p-1)}) \in \mathrm{SL}(p,\mathbb{R})$ satisfies $\psi(g(t)) = 2|t|\sqrt{p/(p-1)} \to \infty$. But $\mathrm{SL}(p,\mathbb{R})$ has Property~(T) for $p \geq 3$ (Theorem~\ref{thm:kazhdan}), so Delorme--Guichardet (Theorem~\ref{thm:delorme-guichardet}) forces every CND function to be bounded --- a contradiction.
\end{proof}
\end{revblock}

\section{Computation of the scatter matrices}\label{sec:supp-computation}

The statistic $-n\log\Lambda^*$ of \eqref{eq:wilks} is finite only when $W$ and $T$ are nonsingular. The matrix $W$ is a sum of $n$ outer products under $g$ centering constraints, so its rank is at most $n - g$, and the rank of $T$ is at most $n$. Our implementation in \CRANpkg{riemtan} passes every computed $W$ and $T$, for the observed labeling and for each relabeling, through the function \code{nearPD} of the R package \CRANpkg{Matrix} with its default tolerances. That function first treats every eigenvalue at or below $10^{-6}$ times the largest as zero, as in the projection of \citet{higham} onto the positive semidefinite matrices. It then raises those eigenvalues to $10^{-8}$ times the largest and rescales rows and columns to restore the diagonal of the projected matrix. A matrix whose eigenvalues all exceed $10^{-6}$ times the largest is returned unchanged, and the statistic is then the $\Lambda^*$ of \eqref{eq:wilks}. When $d > n - g$ the step alters $W$ in every computation, and $T$ as well when $d > n$. The determinants then carry a factor of approximately $10^{-8}$ times the largest eigenvalue of the matrix for each null direction. The computed statistic is therefore a regularized form of $\Lambda^*$ whose value depends on these tolerances. Theorem~\ref{thm:asymptotic} and Propositions~\ref{prop:noncentrality} and~\ref{prop:noncentrality_comparison} hold for fixed $d$ as $n \to \infty$ and do not describe it. The step does not affect the permutation calibration of Section~\ref{sec:permutation}, because the same deterministic map is applied to the observed sample and to every relabeled one. It changes the statistic being calibrated, and the power reported for the configurations named next is the power of the regularized statistic. Every power configuration of the first simulation study, Section~\ref{sec:sim-frechet} ($d \geq 325$ against $n - g = 198$), and the two site comparisons and the cross-study comparison of Section~\ref{sec:application} ($d = 6{,}786$, $n \leq 353$) are in the regime $d > n - g$. The sex comparison ($d = 105$, $n - g = 424$) is not, and it is the only discordance in Section~\ref{sec:application} at which $W$ has full rank. Among the null configurations of the first simulation study, only those at $p = 25$ with $n_l \geq 205$ have $d \leq n - g$. In the second simulation study, Section~\ref{sec:sim-basepoint}, every cell has $d < n - g$. The tangent-space MANOVA comparator of Section~\ref{sec:simulation} is computed with the same step applied to its own scatter matrices.

\section{Additional simulation detail}\label{sec:supp-simulation}

The Type~I error and the power of each test in every cell of the two simulation studies of Section~\ref{sec:simulation}, with confidence intervals, are tabulated in the Supplementary Material. This appendix gives the remaining detail of the second study.

\emph{Second study: Type~I error.}\enspace The Supplementary Material gives the empirical rejection rate of the three exact permutation tests of the second study of Section~\ref{sec:sim-basepoint} in each of its $24$ null cells, at the three attainable sizes of $B = 99$ relabelings. Of the $216$ entries, $11$ lie more than two Monte-Carlo standard errors from the size, against about $10$ expected. Six of them lie above the size, in two adjacent cells, $p = 5$, $g = 3$, balanced, at $s = 1$ and $s = 2$, where the three tests exceed the size together at the $0.10$ level. The other five lie below the size. Those two cells were rerun with $600$ fresh replicates each, with seeds disjoint from the production run. On the fresh replicates the rates at the sizes $0.01$, $0.05$ and $0.10$ are $0.012$, $0.063$ and $0.108$ for $\Lambda^\ast$, $0.018$, $0.065$ and $0.117$ for tangent-space MANOVA at $\hat{C}$ and $0.013$, $0.060$ and $0.113$ at the identity in the cell at $s = 1$, and $0.008$, $0.050$ and $0.127$, $0.012$, $0.052$ and $0.112$, and $0.010$, $0.055$ and $0.112$ in the cell at $s = 2$. Of these $18$ entries, $16$ lie within two standard errors of the size at $R = 600$ and two lie outside, at $2.1$ and $2.2$ standard errors. The $p$-values on the fresh replicates are consistent with uniformity by decile (chi-square $p$-values from $0.61$ to $0.82$). The table in the Supplementary Material shows the production run.

\emph{Second study: power.}\enspace The Supplementary Material lists the empirical power of the three tests in each of the $144$ power cells of the second study, with the gain of $\Lambda^\ast$ over each comparator and its $95\%$ interval; Table~\ref{tab:sim-basepoint} is the mean by $(p, s)$. Figures~\ref{fig:supp-basepoint-p5} and~\ref{fig:supp-basepoint-p10} plot every cell at $p = 5$ and at $p = 10$, in the six designs; the four designs with one displaced group at $p = 10$ are Figure~\ref{fig:sim-basepoint}. In the two triangle designs the three tests have power above $0.99$ from the third value of $\theta$ on at both dimensions.

\begin{figure}[p]
\centering
\includegraphics[width=\linewidth]{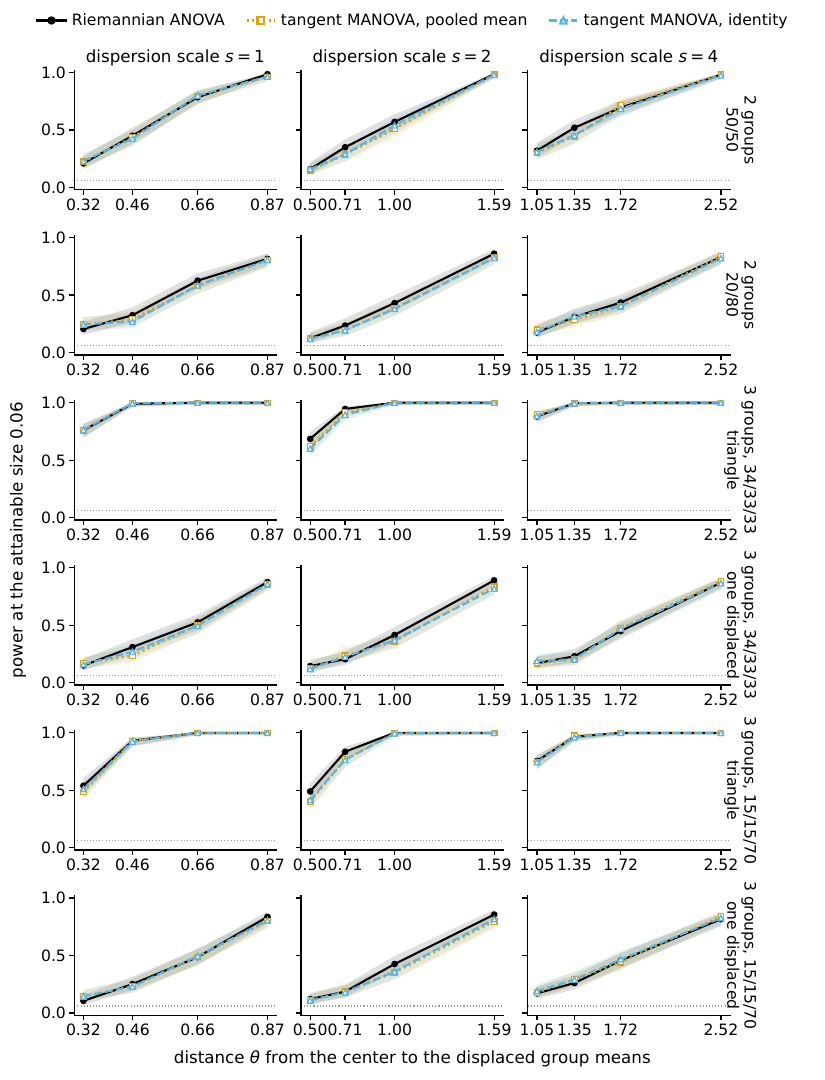}
\caption{Second simulation study at $p = 5$ ($d = 15$, $n = 100$): empirical power at the attainable size $0.06$ of the Riemannian ANOVA (solid, filled circles), tangent-space MANOVA at the pooled Fr\'echet mean (dotted, squares) and at the identity (dashed, triangles), against the distance $\theta$ from the center to the displaced group means, in the six designs (rows) and at the three dispersion scales $s$ (columns). Bands are Wilson $95\%$ confidence intervals over $R = 200$ replicates ($B = 49$); the horizontal dotted line is the attainable size.}
\label{fig:supp-basepoint-p5}
\end{figure}

\begin{figure}[p]
\centering
\includegraphics[width=\linewidth]{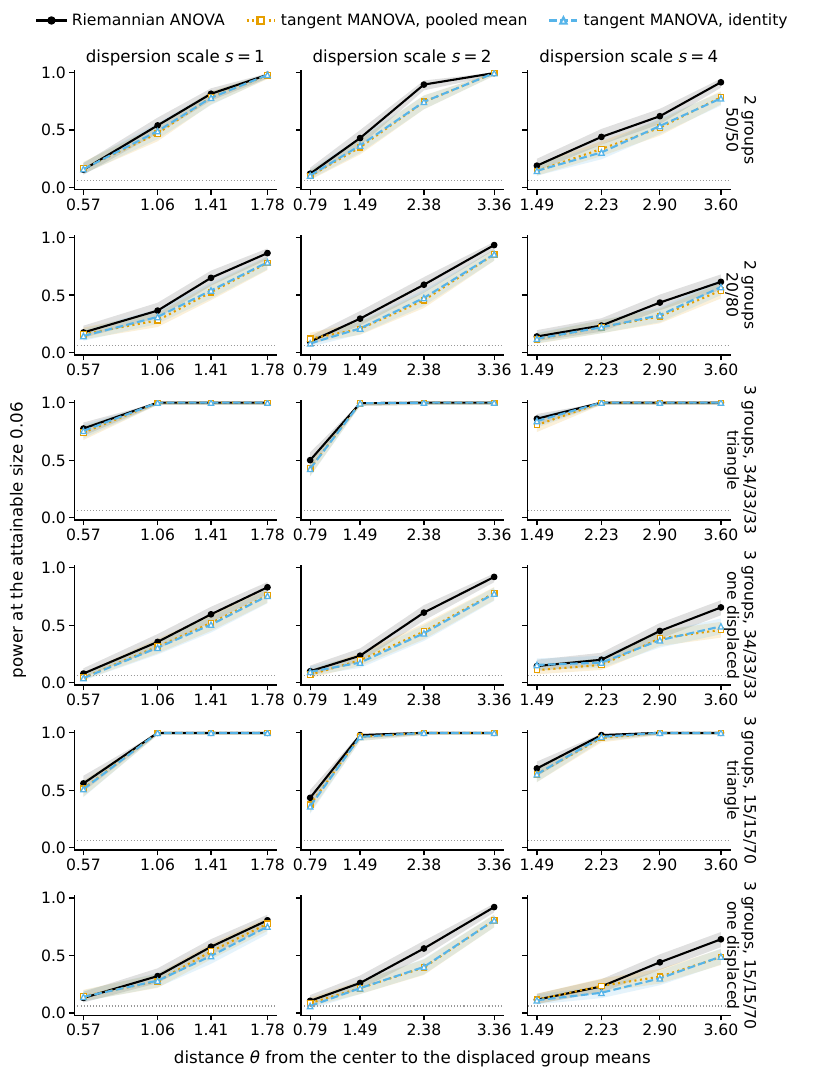}
\caption{Second simulation study at $p = 10$ ($d = 55$, $n = 100$), as in Figure~\ref{fig:supp-basepoint-p5}. The four rows with one displaced group are Figure~\ref{fig:sim-basepoint}.}
\label{fig:supp-basepoint-p10}
\end{figure}

\emph{Second study: base-point diagnostics.}\enspace Each replicate of the second study records four quantities through which the base point enters tangent-space MANOVA, computed on the same sample as the tests. Write $S_{E,l}$ for the sample covariance of the projections $x_{lj}$ of group $l$ at the pooled Fr\'echet mean $\hat{C}$, $\hat{S}_l = \sum_j u_{lj}u_{lj}^T/(n_l - 1)$ for the covariance of the residuals of the proposed construction at the group's own mean, and $\tilde{S}_l$ for $\hat{S}_l$ parallel-transported to $\hat{C}$, so that $S_{E,l}$ and $\tilde{S}_l$ are expressed in the same coordinates. The first diagnostic is the mean over group pairs of $\|S_{E,l} - S_{E,l'}\|_F / \|\tilde{S}_l - \tilde{S}_{l'}\|_F$. The second is the pair $\log|W_E| - \log|W|$ and $\log|T_E| - \log|T|$ and their difference, which is $\log\Lambda_E - \log\Lambda^\ast$. The third is the fraction of replicates on which the decisions of the two tangent-space tests at the $0.05$ level differ, and the corresponding fraction for $\Lambda^\ast$ against tangent-space MANOVA at $\hat{C}$. The fourth is the largest distance $d(\hat{C}_l, \hat{C})$ in units of the within-group root mean squared distance of the observations from their group mean. Table~\ref{tab:sim-basepoint-diag} summarizes the first three by arm, $p$ and $s$. Figure~\ref{fig:supp-basepoint-diag} plots the first against the fourth for all $168$ cells. The per-cell median of the first diagnostic is below one in every cell, with minimum $0.86$, and it falls with $\theta$ within $34$ of the $36$ power designs. The within-group determinant difference and the gap between the two log ratios grow in magnitude with $p$ and $s$; the total-scatter difference does not follow one direction. In the null arm at $p = 10$ it is positive. The decisions of the two tangent-space tests differ on $2$ to $7\%$ of the replicates on average and on up to $13.5\%$ in a cell; those of $\Lambda^\ast$ and tangent-space MANOVA at $\hat{C}$ differ on $2$ to $14\%$ on average and on up to $29.5\%$ in a cell.

\emph{Second study: the diagnostics and the gain.}\enspace Two of these quantities were examined as explanations of the gain of $\Lambda^\ast$ reported in Section~\ref{sec:sim-basepoint}. The first compares the within-group covariances as tangent-space MANOVA sees them with those of the proposed construction. It is the mean over group pairs of the Frobenius distance between the covariances of the projected group clouds at $\hat{C}$, divided by the Frobenius distance between the group covariances $\hat{S}_l = \sum_j u_{lj}u_{lj}^T/(n_l - 1)$ formed at the group means and parallel-transported to $\hat{C}$. The groups share one covariance up to transport, so a ratio above one would mean that projection at the pooled mean makes the group covariances more unequal than they are. The per-cell median of the ratio is below one in every one of the $168$ cells, with minimum $0.86$. The projected covariances differ from each other by less than the transported ones do. The second compares the two statistics through their determinants. At $p = 10$, over the power cells, $\log|W_E| - \log|W|$ averages $-0.35$, $-0.97$ and $-1.52$ at $s = 1, 2, 4$ and $\log|T_E| - \log|T|$ averages $-0.25$, $-0.60$ and $-0.55$, so projection at $\hat{C}$ lowers both determinants, the within-group one by more as the dispersion grows. On average $\log\Lambda_E$ falls below $\log\Lambda^\ast$ by $0.10$, $0.37$ and $0.97$. Neither quantity follows the gain across cells. The gain at $p = 10$ is the same at $s = 2$ and $s = 4$ while the ratio falls further below one and the gap between the log ratios grows from $0.37$ to $0.97$. At $p = 5$ and $s = 4$, where $\log\Lambda_E$ falls below $\log\Lambda^\ast$ by $0.28$, more than at $p = 10$ and $s = 1$, there is no gain. The diagnostics therefore do not support the explanation that motivated the design, a loss of efficiency of tangent-space MANOVA through unequal projected covariances. They do not supply another. We report the gain as an empirical finding of these designs.

\begin{table}[p]
\centering
\caption{Second simulation study: base-point diagnostics by arm, $p$ and $s$. Columns: the median over cells of the per-cell median ratio of covariance differences, projected at $\hat{C}$ over own-mean and transported, with its range over cells; the means of $\log|W_E| - \log|W|$, $\log|T_E| - \log|T|$ and $\log\Lambda_E - \log\Lambda^\ast$; the mean fraction of replicates on which the decisions at the $0.05$ level of tangent-space MANOVA at $\hat{C}$ and at $\mathbf{I}$ differ, with its maximum over cells; the same for $\Lambda^\ast$ against tangent-space MANOVA at $\hat{C}$. Null arm, $24$ cells at $B = 99$; power arm, $144$ cells at $B = 49$.}
\label{tab:sim-basepoint-diag}
{\footnotesize\setlength{\tabcolsep}{4pt}
\begin{tabular}{@{}ccc c ccc cc cc@{}}
\toprule
Arm & $p$ & $s$ & Ratio, median (range) & $\Delta\log|W|$ & $\Delta\log|T|$ & $\Delta\log\Lambda$ & $\hat{C}$ vs $\mathbf{I}$ & max & $\Lambda^\ast$ vs $\hat{C}$ & max \\
\midrule
Null & $5$ & $1$ & $0.999$ $(0.998,\,1.000)$ & $-0.015$ & $0.008$ & $-0.023$ & $0.018$ & $0.020$ & $0.032$ & $0.040$ \\
Null & $5$ & $2$ & $0.997$ $(0.993,\,0.998)$ & $-0.064$ & $0.004$ & $-0.067$ & $0.018$ & $0.030$ & $0.032$ & $0.050$ \\
Null & $5$ & $4$ & $0.992$ $(0.985,\,0.997)$ & $-0.183$ & $-0.046$ & $-0.137$ & $0.020$ & $0.030$ & $0.041$ & $0.050$ \\
Null & $10$ & $1$ & $0.998$ $(0.997,\,0.999)$ & $-0.068$ & $0.051$ & $-0.119$ & $0.025$ & $0.050$ & $0.022$ & $0.025$ \\
Null & $10$ & $2$ & $0.995$ $(0.990,\,0.998)$ & $-0.233$ & $0.124$ & $-0.357$ & $0.032$ & $0.040$ & $0.031$ & $0.035$ \\
Null & $10$ & $4$ & $0.990$ $(0.979,\,0.996)$ & $-0.542$ & $0.191$ & $-0.734$ & $0.031$ & $0.060$ & $0.039$ & $0.055$ \\
\midrule
Power & $5$ & $1$ & $0.997$ $(0.984,\,0.999)$ & $-0.060$ & $-0.026$ & $-0.034$ & $0.023$ & $0.060$ & $0.075$ & $0.185$ \\
Power & $5$ & $2$ & $0.992$ $(0.954,\,0.999)$ & $-0.144$ & $-0.037$ & $-0.107$ & $0.035$ & $0.110$ & $0.079$ & $0.145$ \\
Power & $5$ & $4$ & $0.980$ $(0.891,\,0.995)$ & $-0.269$ & $0.011$ & $-0.280$ & $0.043$ & $0.085$ & $0.113$ & $0.215$ \\
Power & $10$ & $1$ & $0.996$ $(0.962,\,0.999)$ & $-0.353$ & $-0.249$ & $-0.105$ & $0.038$ & $0.105$ & $0.082$ & $0.205$ \\
Power & $10$ & $2$ & $0.989$ $(0.887,\,0.998)$ & $-0.967$ & $-0.595$ & $-0.372$ & $0.043$ & $0.095$ & $0.095$ & $0.205$ \\
Power & $10$ & $4$ & $0.976$ $(0.859,\,0.994)$ & $-1.524$ & $-0.553$ & $-0.971$ & $0.072$ & $0.135$ & $0.141$ & $0.295$ \\
\bottomrule
\end{tabular}}
\end{table}

\begin{figure}[p]
\centering
\includegraphics[width=\linewidth]{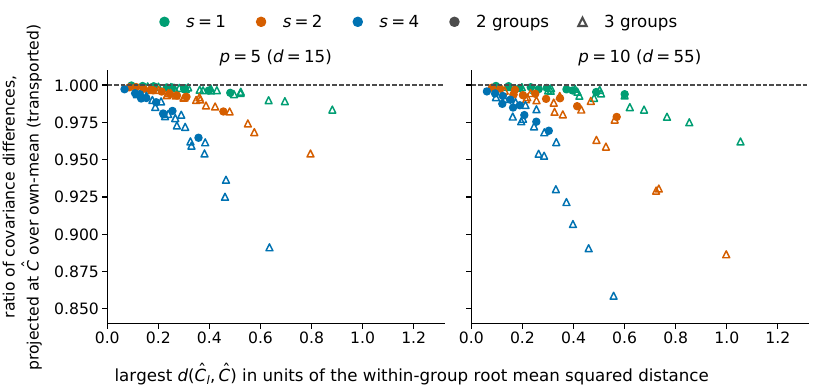}
\caption{Second simulation study: the ratio of covariance differences, projected at $\hat{C}$ over own-mean and transported (per-cell median), against the largest distance from a group mean to the pooled mean in units of the within-group root mean squared distance, for all $168$ cells at $p = 5$ (left) and $p = 10$ (right); color, dispersion scale $s$; filled circles, two groups; open triangles, three groups. The dashed line is a ratio of one. The vertical range does not start at zero.}
\label{fig:supp-basepoint-diag}
\end{figure}

\FloatBarrier

\section{Model diagnostics for the Riemannian normal model}\label{sec:supp-diagnostics}

The test statistic of Section~\ref{sec:wilks} and its null distribution are derived under the Riemannian normal model with a common tangent covariance $\Gamma$. Before applying the tests we examined the adequacy of the Riemannian normal model on the HCP-YA data ($n = 426$ subjects, $14$-network parcellation, tangent dimension $d = 105$). Two diagnostics are informative, and Figure~\ref{fig:supp-diagnostics} reports both. First, Mauchly's test of sphericity of the tangent features is rejected at $p \ll 0.001$, and the off-diagonal correlations among the tangent features are dispersed well beyond zero, with $180$ of the $5{,}460$ pairs exceeding $0.25$ in absolute value, where fewer than one pair is expected under independence at this sample size. The isotropic model of \citet{varo} is therefore inadequate, and the covariance parameter $\Gamma$ of the Riemannian normal model is capturing dependence that is genuinely present. Second, a quantile comparison of the empirical Mahalanobis distances against their $\chi^2_{105}$ reference matches through the bulk of the distribution and departs upward in the extreme right tail, where the five largest distances lie above the $95\%$ pointwise band of the same statistic under the model with the covariance estimated. The Riemannian normal model is thus a first-order approximation to these data and not an exact description, which is why inference in this paper is by permutation.

The figure uses one resting-state scan per subject in the Yeo parcellation. The tangent features are the $\Vect_{\samFr}(\log_{\samFr} C_i)$ of the $426$ connectomes at their sample Fr\'echet mean.

Panel~(A) addresses the covariance structure. Under the isotropic model of \citet{varo}, in which $\Gamma$ is a multiple of the identity, the off-diagonal entries of the correlation matrix of the tangent features would concentrate at zero, with the sampling density of a correlation between two independent features at $n = 426$ drawn as the solid curve. They do not. Of the $5{,}460$ off-diagonal correlations, $180$ exceed $0.25$ in absolute value, where that density puts fewer than one. The largest is $0.90$. Mauchly's test of sphericity on the $105$ features rejects, with $\log_{10} W = -24.8$, a chi-square statistic of $22{,}256.1$ on $d(d+1)/2 - 1 = 5{,}564$ degrees of freedom after the Bartlett correction factor $\rho = 0.917$, and a $p$-value below $10^{-300}$. The ratio of the largest to the smallest eigenvalue of the sample covariance of the features is $227$, against about $9$ expected from sampling variation alone under sphericity at $d/n = 0.25$. The isotropic model is therefore inadequate for these data, and the covariance parameter $\Gamma$ of the Riemannian normal model is capturing dependence that is genuinely present.

Panel~(B) addresses the tangent-space distribution. With $\Gamma$ known, the squared Mahalanobis distances of the tangent features from their Fr\'echet mean are $\chi^2_{105}$ under the model. With the center and covariance estimated from the same $426$ observations the distances are bounded by $(n-1)^2/n$ and their distribution is narrower than $\chi^2_{105}$, so the panel also shows the $95\%$ pointwise band of the sorted distances under the Gaussian model with the covariance estimated, from $1{,}000$ simulated samples of the same size. The quantile comparison follows the reference line through the bulk of the distribution and departs upward in the extreme right tail. The five largest observed distances, $168.7$ to $184.0$, all lie above the band's upper limit, against $\chi^2_{105}$ quantiles of $141.3$ to $154.6$. The Riemannian normal model is thus a first-order description of these data and not an exact one. The departure is confined to the region that a fixed asymptotic reference would use to set critical values, which is why inference throughout the paper is by permutation (Section~\ref{sec:permutation}).

\begin{figure}[t]
\centering
\includegraphics[width=0.98\linewidth]{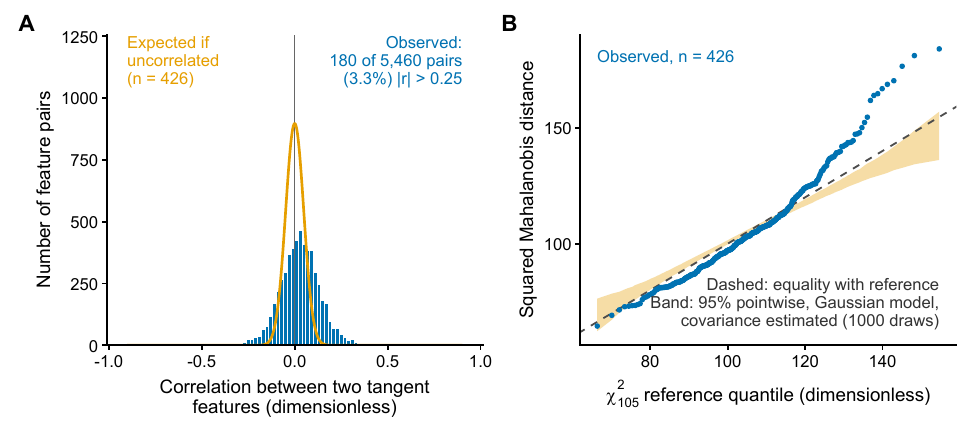}
\caption{The tangent features of the HCP-YA connectomes are not isotropic, and their Mahalanobis distances have a heavier right tail than the Riemannian normal model allows ($n = 426$, $14$-network parcellation, $d = 105$). (A) Histogram of the $5{,}460$ off-diagonal correlations among the $d = 105$ tangent features; the solid curve is the exact sampling density of a correlation between two uncorrelated features at $n = 426$. (B) Sorted squared Mahalanobis distances against $\chi^2_{105}$ quantiles; dashed line, equality; shaded band, $95\%$ pointwise range of the same statistic under the Gaussian model with covariance estimated ($1{,}000$ simulations).}
\label{fig:supp-diagnostics}
\end{figure}

\FloatBarrier

\section{Robustness to the connectivity estimator}\label{sec:supp-estimation}

The conclusions of Section~\ref{sec:application} do not hinge on the choice of estimator. Because edge-level connectivity estimates have well-documented reliability limits \citep{noble2019}, it is worth asking whether the moderate-effect discordances reported there reflect genuine geometric signal or noisy estimation. We therefore re-estimated connectivity under three shrinkage strategies --- oracle-approximating shrinkage \rev{(OAS)} \citep{chen2010}, empirical-Bayes \rev{(EB)} shrinkage toward the grand mean, and empirical-Bayes shrinkage with session-splitting \citep{mejia2018} --- and repeated the affected comparisons (Table~\ref{tab:shrinkage}).

\begin{table}[t]
\centering
\caption{Shrinkage robustness of the moderate-effect discordances. ``Original'' rows reproduce Section~\ref{sec:application}; remaining rows re-estimate connectivity under the named shrinkage method. \protect\rev{Riemannian $p$-values are by permutation with $B = 100$, except the original sex comparison at $B = 20$; Fr\'echet $p$-values are from the asymptotic $\chi^2_1$ reference.} Bold $p$-values are significant at $\alpha = 0.05$.}
\label{tab:shrinkage}
\begin{tabular}{lllrrr}
\toprule
Comparison & Connectivity & $n$ & Fr\'echet $F$ & Fr\'echet $p$ & Riemannian $p$ \\
\midrule
\multicolumn{6}{l}{\textit{Sex differences (Yeo 14 parcellation)}} \\
\quad Original       & Pearson         & 203/223 & 0.964 & 0.326 & \textbf{\rev{$<0.05$}} \\
\quad OAS shrinkage  & OAS             & 203/223 & 2.101 & 0.147 & \textbf{\rev{$<0.01$}} \\
\quad EB session-split & EB (split-half)& 203/223 & 0.965 & 0.326 & \textbf{\rev{$<0.01$}} \\
\midrule
\multicolumn{6}{l}{\textit{Site effect (Schaefer 100 parcellation)}} \\
\quad Original       & Pre-computed    & 205/148 & 1.491 & 0.222 & \textbf{\rev{$<0.01$}} \\
\quad EB grand mean  & EB (grand mean) & 205/148 & 0.284 & 0.594 & 0.770 \\
\bottomrule
\end{tabular}
\end{table}

The sex discordance is robust to all three strategies: whether connectivity is estimated by raw Pearson correlation, oracle-approximating shrinkage, or empirical-Bayes session-splitting, the Riemannian test detects a difference \rev{at the resolution of its permutation run ($p < 0.05$ at $B = 20$; $p < 0.01$ at $B = 100$)} while Fr\'echet ANOVA does not. The negligible effect of session-split shrinkage is expected, since the favorable ratio of time points to regions ($T/p \approx 85$ in the Yeo-14 parcellation) gives a shrinkage intensity near zero. The site discordance, by contrast, does not survive empirical-Bayes shrinkage toward the grand mean (mean intensity $\bar\lambda = 0.89$): after shrinkage both tests return a null. The cleanest reading is that the sex difference \rev{persists under reliability correction}, while the full-sample site effect, \rev{detected at the raw estimate, appears to be} a smaller, noise-sensitive signal that aggressive shrinkage damps away. A complementary sample-size scaling series under the null, confirming that neither test drifts toward spurious detection as $n$ grows, is reported in Appendix~\ref{sec:supp-scaling}.

\section{Sample-size scaling under the null}\label{sec:supp-scaling}

To check that neither test produces spurious detections as the sample size grows, the REST1-versus-REST2 negative control of Section~\ref{sec:application} was repeated at $n = 10, 20, 40$ and $80$ (HCP-Aging, Schaefer-100, $116 \times 116$ matrices, 100 permutations). Table~\ref{tab:scaling} shows no trend toward significance in either test at any size; the Riemannian ANOVA behaves stably, as a well-calibrated test should.

\begin{table}[t]
\centering
\caption{Sample-size scaling series for the REST1 versus REST2 negative control (HCP-Aging, Schaefer-100, 100 permutations). Neither test produces a significant result at any sample size.}
\label{tab:scaling}
\begin{tabular}{rrrr}
\toprule
$n$ & Fr\'echet $F$ & Fr\'echet $p$ & Riemannian $p$ \\
\midrule
10 & 1.027 & 0.311 & 0.670 \\
20 & 0.720 & 0.396 & 0.670 \\
40 & 1.336 & 0.248 & 1.000 \\
80 & 0.260 & 0.610 & 0.990 \\
\bottomrule
\end{tabular}
\end{table}

\section{Code and data availability}\label{sec:repro}

The first simulation study of Section~\ref{sec:simulation} is fully reproducible from a fixed random seed. Data are generated from a Riemannian normal law whose parameters are calibrated to real connectivity data; the connectome cohorts used for calibration (HCP-YA and HCP-Aging) are access-controlled under the Human Connectome Project Data Use Terms and are not redistributed. We instead release the derived calibration object --- the sample Fr\'echet mean, the oracle-approximating-shrinkage estimate of the tangent covariance, the residual scale, and the shift direction at each parcellation size --- together with all analysis code. Given access to the cohorts the calibration is regenerated from the raw connectomes; given the calibration object, every reported power and Type~I value, with its Wilson confidence interval, is regenerated exactly from the recorded base seed. Each replicate's seed is a deterministic function of that base seed and the replicate's position in the design, so individual cells reproduce replicate for replicate. The simulation was run on a high-performance cluster under R~4.5.1 with \texttt{riemtan}~0.3.0 and \texttt{riemstats}~0.2.0. The sample Fr\'echet means were computed with the solver of \texttt{riemtan} at a relative convergence tolerance of $10^{-4}$. The scatter matrices were regularized with the function \code{nearPD} of \CRANpkg{Matrix}~1.7-5 at its default tolerances, as described in Appendix~\ref{sec:supp-computation}. The full design took about $41{,}000$ core-hours. A full session record, the per-cell seed map, and the job provenance accompany the code.

The second simulation study is reproducible in the same way from its own base seed, each replicate's seed being the base seed plus a per-cell offset plus the replicate index; the calibration object at $p = 5$ and $p = 10$ is released with the code. It was run on the same cluster under R~4.5.1 with \texttt{riemtan}~0.3.0, \texttt{riemstats}~0.2.0 and \CRANpkg{Matrix}~1.7-5 at its default tolerances, with single-threaded BLAS. Because $\Lambda^\ast$ recomputes $g$ Fr\'echet means for each of the $B$ relabelings, the means inside the permutation loop were computed by a damped Newton iteration for the Karcher mean written for the study, stopped when the gradient norm fell below $10^{-8}$ in the Riemannian norm; on a test cell at $p = 10$ and $s = 4$ its means agree with those of the \texttt{riemtan} solver run to convergence to within $2 \times 10^{-10}$ in geodesic distance. The resulting $\log\Lambda^\ast$ agrees with that of \texttt{riemstats} to within $3 \times 10^{-10}$ on the observed labeling and on permuted ones. The $\theta$ grid was set by a pilot of $30$ cells. The production run of $168$ cells, $33{,}600$ replicates, took $83$ core-hours; the pilot, the two rerun cells and the calibration took a further $6$. At $s = 4$ the exponential map of the sampler projects a draw onto the positive definite matrices whenever an eigenvalue falls below $10^{-6}$ of the largest; this touched $0.1\%$ of the draws at $p = 5$ and $1.1\%$ at $p = 10$, and none at $s \leq 2$. All three tests see the same draws.

\end{appendix}

\begin{acks}[Acknowledgments]
The authors thank the CONNplexity Lab at Purdue University for the preprocessed HCP-YA timeseries.
\end{acks}

\begin{funding}
This research was partly supported by NIH grant R01 NS112303.
\end{funding}

\begin{supplement}
\stitle{Per-cell tables of the simulation studies}
\sdescription{Empirical Type~I error and power of each test in every cell of the two simulation studies of Section~\ref{sec:simulation}, with confidence intervals. In this version they are printed after the references, as Tables~S1--S4.}
\end{supplement}

\begin{supplement}
\stitle{Code and data}
\sdescription{The analysis code, the calibration objects and the seed maps described in Appendix~\ref{sec:repro}. The R packages \texttt{riemtan} and \texttt{riemstats} implement the method.}
\end{supplement}

\section*{Statement on AI use}
Artificial intelligence was used in three parts of writing this document. First, in writing part of the code whose execution led to the experimental results presented in the article, as well as the figures and tables presented in it. Having said that, at the core of these experiments are two R libraries for which I wrote most of the code, and I carried out the design of the experiments.

Second, the abstract, the introduction, and other parts of the text were originally written in Spanish and translated into English with the help of artificial intelligence. For the sections describing the experiments and their results, the artificial intelligence was asked for an initial draft, which I edited extensively.

Finally, some details of the logarithmic image difference formula test found in one of the appendices were suggested by artificial intelligence.

\bibliographystyle{imsart-nameyear}
\bibliography{biblio}

\clearpage
\section*{Supplementary tables}
\setcounter{table}{0}
\renewcommand{\thetable}{S\arabic{table}}

\emph{First study: Type~I error.}\enspace Table~\ref{tab:sim-typeI-full} gives the empirical rejection rate and its Wilson $95\%$ confidence interval for every $(p, n_l, \alpha)$ cell of the Type~I calibration summarized in Section~5.1, for each of the three tests. The attainable sizes of a permutation test with $B = 49$ relabelings and the $p$-value convention of Section~4 are $0.02$, $0.06$ and $0.10$ at the nominal levels $0.01$, $0.05$ and $0.10$. The Monte-Carlo standard error in the last column is that of a binomial proportion at the nominal level over $R = 200$ replicates.

{\footnotesize\setlength{\tabcolsep}{3.5pt}
\begin{longtable}{lcccccc}
\caption{Empirical Type~I error of the three permutation tests ($B = 49$) under the null, both groups drawn from $\mathrm{RiemN}(\samFr,\widehat{\Gamma})$ fitted to connectome data ($g=2$, $R=200$ per cell), with Wilson $95\%$ confidence intervals for every $(p,n_l,\alpha)$ cell. The attainable sizes at the nominal levels $0.01$, $0.05$ and $0.10$ are $0.02$, $0.06$ and $0.10$. The last column is the Monte-Carlo standard error of a proportion at the nominal level.}\label{tab:sim-typeI-full}\\
\toprule
$(p,g)$ & $n_l$ & nominal $\alpha$ & Riemannian ANOVA & Fr\'echet ANOVA & Tangent MANOVA & MC std.\ err. \\
\midrule
\endfirsthead
\multicolumn{7}{l}{\small\itshape Table~\ref{tab:sim-typeI-full} continued}\\
\toprule
$(p,g)$ & $n_l$ & nominal $\alpha$ & Riemannian ANOVA & Fr\'echet ANOVA & Tangent MANOVA & MC std.\ err. \\
\midrule
\endhead
\midrule \multicolumn{7}{r}{\small\itshape continued on next page}\\
\endfoot
\bottomrule
\endlastfoot
$(25,2)$ & 20 & 0.01 & 0.015 $[0.005,\,0.043]$ & 0.020 $[0.008,\,0.050]$ & 0.015 $[0.005,\,0.043]$ & 0.0070 \\
$(25,2)$ & 20 & 0.05 & 0.075 $[0.046,\,0.120]$ & 0.070 $[0.042,\,0.114]$ & 0.060 $[0.035,\,0.102]$ & 0.0154 \\
$(25,2)$ & 20 & 0.10 & 0.100 $[0.066,\,0.149]$ & 0.115 $[0.078,\,0.167]$ & 0.090 $[0.058,\,0.138]$ & 0.0212 \\
$(25,2)$ & 50 & 0.01 & 0.010 $[0.003,\,0.036]$ & 0.015 $[0.005,\,0.043]$ & 0.015 $[0.005,\,0.043]$ & 0.0070 \\
$(25,2)$ & 50 & 0.05 & 0.040 $[0.020,\,0.077]$ & 0.055 $[0.031,\,0.096]$ & 0.055 $[0.031,\,0.096]$ & 0.0154 \\
$(25,2)$ & 50 & 0.10 & 0.070 $[0.042,\,0.114]$ & 0.090 $[0.058,\,0.138]$ & 0.075 $[0.046,\,0.120]$ & 0.0212 \\
$(25,2)$ & 100 & 0.01 & 0.010 $[0.003,\,0.036]$ & 0.030 $[0.014,\,0.064]$ & 0.020 $[0.008,\,0.050]$ & 0.0070 \\
$(25,2)$ & 100 & 0.05 & 0.065 $[0.038,\,0.108]$ & 0.090 $[0.058,\,0.138]$ & 0.065 $[0.038,\,0.108]$ & 0.0154 \\
$(25,2)$ & 100 & 0.10 & 0.115 $[0.078,\,0.167]$ & 0.120 $[0.082,\,0.172]$ & 0.110 $[0.074,\,0.161]$ & 0.0212 \\
$(25,2)$ & 150 & 0.01 & 0.035 $[0.017,\,0.070]$ & 0.015 $[0.005,\,0.043]$ & 0.025 $[0.011,\,0.057]$ & 0.0070 \\
$(25,2)$ & 150 & 0.05 & 0.085 $[0.054,\,0.132]$ & 0.030 $[0.014,\,0.064]$ & 0.080 $[0.050,\,0.126]$ & 0.0154 \\
$(25,2)$ & 150 & 0.10 & 0.105 $[0.070,\,0.155]$ & 0.060 $[0.035,\,0.102]$ & 0.100 $[0.066,\,0.149]$ & 0.0212 \\
$(25,2)$ & 205 & 0.01 & 0.025 $[0.011,\,0.057]$ & 0.005 $[0.001,\,0.028]$ & 0.025 $[0.011,\,0.057]$ & 0.0070 \\
$(25,2)$ & 205 & 0.05 & 0.090 $[0.058,\,0.138]$ & 0.030 $[0.014,\,0.064]$ & 0.075 $[0.046,\,0.120]$ & 0.0154 \\
$(25,2)$ & 205 & 0.10 & 0.120 $[0.082,\,0.172]$ & 0.080 $[0.050,\,0.126]$ & 0.105 $[0.070,\,0.155]$ & 0.0212 \\
$(25,2)$ & 213 & 0.01 & 0.015 $[0.005,\,0.043]$ & 0.020 $[0.008,\,0.050]$ & 0.010 $[0.003,\,0.036]$ & 0.0070 \\
$(25,2)$ & 213 & 0.05 & 0.030 $[0.014,\,0.064]$ & 0.060 $[0.035,\,0.102]$ & 0.040 $[0.020,\,0.077]$ & 0.0154 \\
$(25,2)$ & 213 & 0.10 & 0.040 $[0.020,\,0.077]$ & 0.085 $[0.054,\,0.132]$ & 0.060 $[0.035,\,0.102]$ & 0.0212 \\
$(30,2)$ & 20 & 0.01 & 0.035 $[0.017,\,0.070]$ & 0.035 $[0.017,\,0.070]$ & 0.035 $[0.017,\,0.070]$ & 0.0070 \\
$(30,2)$ & 20 & 0.05 & 0.065 $[0.038,\,0.108]$ & 0.095 $[0.062,\,0.144]$ & 0.085 $[0.054,\,0.132]$ & 0.0154 \\
$(30,2)$ & 20 & 0.10 & 0.115 $[0.078,\,0.167]$ & 0.130 $[0.090,\,0.184]$ & 0.130 $[0.090,\,0.184]$ & 0.0212 \\
$(30,2)$ & 50 & 0.01 & 0.025 $[0.011,\,0.057]$ & 0.025 $[0.011,\,0.057]$ & 0.015 $[0.005,\,0.043]$ & 0.0070 \\
$(30,2)$ & 50 & 0.05 & 0.070 $[0.042,\,0.114]$ & 0.050 $[0.027,\,0.090]$ & 0.060 $[0.035,\,0.102]$ & 0.0154 \\
$(30,2)$ & 50 & 0.10 & 0.105 $[0.070,\,0.155]$ & 0.105 $[0.070,\,0.155]$ & 0.100 $[0.066,\,0.149]$ & 0.0212 \\
$(30,2)$ & 100 & 0.01 & 0.015 $[0.005,\,0.043]$ & 0.020 $[0.008,\,0.050]$ & 0.020 $[0.008,\,0.050]$ & 0.0070 \\
$(30,2)$ & 100 & 0.05 & 0.055 $[0.031,\,0.096]$ & 0.065 $[0.038,\,0.108]$ & 0.060 $[0.035,\,0.102]$ & 0.0154 \\
$(30,2)$ & 100 & 0.10 & 0.075 $[0.046,\,0.120]$ & 0.100 $[0.066,\,0.149]$ & 0.085 $[0.054,\,0.132]$ & 0.0212 \\
$(30,2)$ & 150 & 0.01 & 0.025 $[0.011,\,0.057]$ & 0.010 $[0.003,\,0.036]$ & 0.015 $[0.005,\,0.043]$ & 0.0070 \\
$(30,2)$ & 150 & 0.05 & 0.075 $[0.046,\,0.120]$ & 0.075 $[0.046,\,0.120]$ & 0.055 $[0.031,\,0.096]$ & 0.0154 \\
$(30,2)$ & 150 & 0.10 & 0.120 $[0.082,\,0.172]$ & 0.120 $[0.082,\,0.172]$ & 0.100 $[0.066,\,0.149]$ & 0.0212 \\
$(30,2)$ & 205 & 0.01 & 0.025 $[0.011,\,0.057]$ & 0.010 $[0.003,\,0.036]$ & 0.020 $[0.008,\,0.050]$ & 0.0070 \\
$(30,2)$ & 205 & 0.05 & 0.055 $[0.031,\,0.096]$ & 0.085 $[0.054,\,0.132]$ & 0.060 $[0.035,\,0.102]$ & 0.0154 \\
$(30,2)$ & 205 & 0.10 & 0.100 $[0.066,\,0.149]$ & 0.125 $[0.086,\,0.178]$ & 0.130 $[0.090,\,0.184]$ & 0.0212 \\
$(30,2)$ & 213 & 0.01 & 0.015 $[0.005,\,0.043]$ & 0.025 $[0.011,\,0.057]$ & 0.030 $[0.014,\,0.064]$ & 0.0070 \\
$(30,2)$ & 213 & 0.05 & 0.065 $[0.038,\,0.108]$ & 0.050 $[0.027,\,0.090]$ & 0.065 $[0.038,\,0.108]$ & 0.0154 \\
$(30,2)$ & 213 & 0.10 & 0.095 $[0.062,\,0.144]$ & 0.120 $[0.082,\,0.172]$ & 0.140 $[0.099,\,0.195]$ & 0.0212 \\
$(35,2)$ & 20 & 0.01 & 0.020 $[0.008,\,0.050]$ & 0.035 $[0.017,\,0.070]$ & 0.010 $[0.003,\,0.036]$ & 0.0070 \\
$(35,2)$ & 20 & 0.05 & 0.060 $[0.035,\,0.102]$ & 0.085 $[0.054,\,0.132]$ & 0.040 $[0.020,\,0.077]$ & 0.0154 \\
$(35,2)$ & 20 & 0.10 & 0.090 $[0.058,\,0.138]$ & 0.140 $[0.099,\,0.195]$ & 0.070 $[0.042,\,0.114]$ & 0.0212 \\
$(35,2)$ & 50 & 0.01 & 0.015 $[0.005,\,0.043]$ & 0.025 $[0.011,\,0.057]$ & 0.015 $[0.005,\,0.043]$ & 0.0070 \\
$(35,2)$ & 50 & 0.05 & 0.060 $[0.035,\,0.102]$ & 0.060 $[0.035,\,0.102]$ & 0.050 $[0.027,\,0.090]$ & 0.0154 \\
$(35,2)$ & 50 & 0.10 & 0.090 $[0.058,\,0.138]$ & 0.085 $[0.054,\,0.132]$ & 0.065 $[0.038,\,0.108]$ & 0.0212 \\
$(35,2)$ & 100 & 0.01 & 0.010 $[0.003,\,0.036]$ & 0.020 $[0.008,\,0.050]$ & 0.005 $[0.001,\,0.028]$ & 0.0070 \\
$(35,2)$ & 100 & 0.05 & 0.050 $[0.027,\,0.090]$ & 0.045 $[0.024,\,0.083]$ & 0.045 $[0.024,\,0.083]$ & 0.0154 \\
$(35,2)$ & 100 & 0.10 & 0.095 $[0.062,\,0.144]$ & 0.065 $[0.038,\,0.108]$ & 0.080 $[0.050,\,0.126]$ & 0.0212 \\
$(35,2)$ & 150 & 0.01 & 0.030 $[0.014,\,0.064]$ & 0.025 $[0.011,\,0.057]$ & 0.015 $[0.005,\,0.043]$ & 0.0070 \\
$(35,2)$ & 150 & 0.05 & 0.065 $[0.038,\,0.108]$ & 0.070 $[0.042,\,0.114]$ & 0.065 $[0.038,\,0.108]$ & 0.0154 \\
$(35,2)$ & 150 & 0.10 & 0.115 $[0.078,\,0.167]$ & 0.105 $[0.070,\,0.155]$ & 0.120 $[0.082,\,0.172]$ & 0.0212 \\
$(35,2)$ & 205 & 0.01 & 0.025 $[0.011,\,0.057]$ & 0.040 $[0.020,\,0.077]$ & 0.015 $[0.005,\,0.043]$ & 0.0070 \\
$(35,2)$ & 205 & 0.05 & 0.050 $[0.027,\,0.090]$ & 0.080 $[0.050,\,0.126]$ & 0.065 $[0.038,\,0.108]$ & 0.0154 \\
$(35,2)$ & 205 & 0.10 & 0.080 $[0.050,\,0.126]$ & 0.120 $[0.082,\,0.172]$ & 0.095 $[0.062,\,0.144]$ & 0.0212 \\
$(35,2)$ & 213 & 0.01 & 0.005 $[0.001,\,0.028]$ & 0.015 $[0.005,\,0.043]$ & 0.015 $[0.005,\,0.043]$ & 0.0070 \\
$(35,2)$ & 213 & 0.05 & 0.065 $[0.038,\,0.108]$ & 0.055 $[0.031,\,0.096]$ & 0.065 $[0.038,\,0.108]$ & 0.0154 \\
$(35,2)$ & 213 & 0.10 & 0.085 $[0.054,\,0.132]$ & 0.100 $[0.066,\,0.149]$ & 0.110 $[0.074,\,0.161]$ & 0.0212 \\
$(40,2)$ & 20 & 0.01 & 0.015 $[0.005,\,0.043]$ & 0.015 $[0.005,\,0.043]$ & 0.010 $[0.003,\,0.036]$ & 0.0070 \\
$(40,2)$ & 20 & 0.05 & 0.070 $[0.042,\,0.114]$ & 0.060 $[0.035,\,0.102]$ & 0.075 $[0.046,\,0.120]$ & 0.0154 \\
$(40,2)$ & 20 & 0.10 & 0.120 $[0.082,\,0.172]$ & 0.090 $[0.058,\,0.138]$ & 0.100 $[0.066,\,0.149]$ & 0.0212 \\
$(40,2)$ & 50 & 0.01 & 0.025 $[0.011,\,0.057]$ & 0.010 $[0.003,\,0.036]$ & 0.030 $[0.014,\,0.064]$ & 0.0070 \\
$(40,2)$ & 50 & 0.05 & 0.065 $[0.038,\,0.108]$ & 0.060 $[0.035,\,0.102]$ & 0.060 $[0.035,\,0.102]$ & 0.0154 \\
$(40,2)$ & 50 & 0.10 & 0.110 $[0.074,\,0.161]$ & 0.105 $[0.070,\,0.155]$ & 0.100 $[0.066,\,0.149]$ & 0.0212 \\
$(40,2)$ & 100 & 0.01 & 0.020 $[0.008,\,0.050]$ & 0.035 $[0.017,\,0.070]$ & 0.015 $[0.005,\,0.043]$ & 0.0070 \\
$(40,2)$ & 100 & 0.05 & 0.065 $[0.038,\,0.108]$ & 0.070 $[0.042,\,0.114]$ & 0.040 $[0.020,\,0.077]$ & 0.0154 \\
$(40,2)$ & 100 & 0.10 & 0.080 $[0.050,\,0.126]$ & 0.120 $[0.082,\,0.172]$ & 0.065 $[0.038,\,0.108]$ & 0.0212 \\
$(40,2)$ & 150 & 0.01 & 0.020 $[0.008,\,0.050]$ & 0.015 $[0.005,\,0.043]$ & 0.020 $[0.008,\,0.050]$ & 0.0070 \\
$(40,2)$ & 150 & 0.05 & 0.050 $[0.027,\,0.090]$ & 0.055 $[0.031,\,0.096]$ & 0.055 $[0.031,\,0.096]$ & 0.0154 \\
$(40,2)$ & 150 & 0.10 & 0.110 $[0.074,\,0.161]$ & 0.100 $[0.066,\,0.149]$ & 0.095 $[0.062,\,0.144]$ & 0.0212 \\
$(40,2)$ & 205 & 0.01 & 0.035 $[0.017,\,0.070]$ & 0.000 $[0.000,\,0.019]$ & 0.025 $[0.011,\,0.057]$ & 0.0070 \\
$(40,2)$ & 205 & 0.05 & 0.090 $[0.058,\,0.138]$ & 0.040 $[0.020,\,0.077]$ & 0.090 $[0.058,\,0.138]$ & 0.0154 \\
$(40,2)$ & 205 & 0.10 & 0.115 $[0.078,\,0.167]$ & 0.090 $[0.058,\,0.138]$ & 0.120 $[0.082,\,0.172]$ & 0.0212 \\
$(40,2)$ & 213 & 0.01 & 0.010 $[0.003,\,0.036]$ & 0.020 $[0.008,\,0.050]$ & 0.015 $[0.005,\,0.043]$ & 0.0070 \\
$(40,2)$ & 213 & 0.05 & 0.035 $[0.017,\,0.070]$ & 0.095 $[0.062,\,0.144]$ & 0.045 $[0.024,\,0.083]$ & 0.0154 \\
$(40,2)$ & 213 & 0.10 & 0.065 $[0.038,\,0.108]$ & 0.150 $[0.107,\,0.206]$ & 0.055 $[0.031,\,0.096]$ & 0.0212 \\
\end{longtable}}

\emph{First study: power.}\enspace Table~\ref{tab:sim-power-full} lists the empirical power of the three tests at every $(p, \theta)$ of Figure~1, with Wilson $95\%$ confidence intervals. The Riemannian ANOVA and tangent-space MANOVA differ by at most $0.045$ at any point, with a mean difference of $-0.002$ over the $40$ configurations with $\theta > 0$. The rows at $\theta = 0$ are null configurations at $n_l = 100$ generated with their own seeds; they agree with the corresponding rows of Table~\ref{tab:sim-typeI-full} within Monte-Carlo error for all three tests.

{\footnotesize\setlength{\tabcolsep}{3.5pt}
\begin{longtable}{ccccc}
\caption{Empirical power of the three permutation tests ($B = 49$) against a location shift of Riemannian size $\theta$, $n_l = 100$ per group, $R = 200$ replicates per cell, with Wilson $95\%$ confidence intervals. These are the values plotted in Figure~1.}\label{tab:sim-power-full}\\
\toprule
$p$ & $\theta$ & Riemannian ANOVA & Fr\'echet ANOVA & Tangent MANOVA \\
\midrule
\endfirsthead
\multicolumn{5}{l}{\small\itshape Table~\ref{tab:sim-power-full} continued}\\
\toprule
$p$ & $\theta$ & Riemannian ANOVA & Fr\'echet ANOVA & Tangent MANOVA \\
\midrule
\endhead
\midrule \multicolumn{5}{r}{\small\itshape continued on next page}\\
\endfoot
\bottomrule
\endlastfoot
25 & 0.0 & 0.070 $[0.042,\,0.114]$ & 0.050 $[0.027,\,0.090]$ & 0.065 $[0.038,\,0.108]$ \\
25 & 0.2 & 0.160 $[0.116,\,0.217]$ & 0.040 $[0.020,\,0.077]$ & 0.160 $[0.116,\,0.217]$ \\
25 & 0.4 & 0.450 $[0.383,\,0.519]$ & 0.090 $[0.058,\,0.138]$ & 0.480 $[0.412,\,0.549]$ \\
25 & 0.6 & 0.820 $[0.761,\,0.867]$ & 0.095 $[0.062,\,0.144]$ & 0.830 $[0.772,\,0.876]$ \\
25 & 0.8 & 0.960 $[0.923,\,0.980]$ & 0.260 $[0.204,\,0.325]$ & 0.965 $[0.930,\,0.983]$ \\
25 & 1.0 & 0.990 $[0.964,\,0.997]$ & 0.455 $[0.387,\,0.524]$ & 1.000 $[0.981,\,1.000]$ \\
25 & 1.2 & 1.000 $[0.981,\,1.000]$ & 0.725 $[0.659,\,0.782]$ & 1.000 $[0.981,\,1.000]$ \\
25 & 1.4 & 1.000 $[0.981,\,1.000]$ & 0.945 $[0.904,\,0.969]$ & 1.000 $[0.981,\,1.000]$ \\
25 & 1.6 & 1.000 $[0.981,\,1.000]$ & 0.985 $[0.957,\,0.995]$ & 1.000 $[0.981,\,1.000]$ \\
25 & 1.8 & 1.000 $[0.981,\,1.000]$ & 1.000 $[0.981,\,1.000]$ & 1.000 $[0.981,\,1.000]$ \\
25 & 2.0 & 1.000 $[0.981,\,1.000]$ & 1.000 $[0.981,\,1.000]$ & 1.000 $[0.981,\,1.000]$ \\
30 & 0.0 & 0.025 $[0.011,\,0.057]$ & 0.060 $[0.035,\,0.102]$ & 0.025 $[0.011,\,0.057]$ \\
30 & 0.2 & 0.170 $[0.124,\,0.228]$ & 0.030 $[0.014,\,0.064]$ & 0.170 $[0.124,\,0.228]$ \\
30 & 0.4 & 0.395 $[0.330,\,0.464]$ & 0.065 $[0.038,\,0.108]$ & 0.395 $[0.330,\,0.464]$ \\
30 & 0.6 & 0.790 $[0.728,\,0.841]$ & 0.090 $[0.058,\,0.138]$ & 0.765 $[0.702,\,0.818]$ \\
30 & 0.8 & 0.935 $[0.892,\,0.962]$ & 0.180 $[0.133,\,0.239]$ & 0.920 $[0.874,\,0.950]$ \\
30 & 1.0 & 0.990 $[0.964,\,0.997]$ & 0.350 $[0.287,\,0.418]$ & 0.990 $[0.964,\,0.997]$ \\
30 & 1.2 & 1.000 $[0.981,\,1.000]$ & 0.575 $[0.506,\,0.641]$ & 1.000 $[0.981,\,1.000]$ \\
30 & 1.4 & 1.000 $[0.981,\,1.000]$ & 0.810 $[0.750,\,0.858]$ & 1.000 $[0.981,\,1.000]$ \\
30 & 1.6 & 1.000 $[0.981,\,1.000]$ & 0.950 $[0.910,\,0.973]$ & 1.000 $[0.981,\,1.000]$ \\
30 & 1.8 & 1.000 $[0.981,\,1.000]$ & 0.995 $[0.972,\,0.999]$ & 1.000 $[0.981,\,1.000]$ \\
30 & 2.0 & 1.000 $[0.981,\,1.000]$ & 1.000 $[0.981,\,1.000]$ & 1.000 $[0.981,\,1.000]$ \\
35 & 0.0 & 0.050 $[0.027,\,0.090]$ & 0.055 $[0.031,\,0.096]$ & 0.060 $[0.035,\,0.102]$ \\
35 & 0.2 & 0.130 $[0.090,\,0.184]$ & 0.040 $[0.020,\,0.077]$ & 0.135 $[0.094,\,0.189]$ \\
35 & 0.4 & 0.345 $[0.283,\,0.413]$ & 0.080 $[0.050,\,0.126]$ & 0.325 $[0.264,\,0.393]$ \\
35 & 0.6 & 0.665 $[0.597,\,0.727]$ & 0.095 $[0.062,\,0.144]$ & 0.660 $[0.592,\,0.722]$ \\
35 & 0.8 & 0.850 $[0.794,\,0.893]$ & 0.160 $[0.116,\,0.217]$ & 0.845 $[0.788,\,0.889]$ \\
35 & 1.0 & 0.950 $[0.910,\,0.973]$ & 0.235 $[0.182,\,0.298]$ & 0.965 $[0.930,\,0.983]$ \\
35 & 1.2 & 0.995 $[0.972,\,0.999]$ & 0.355 $[0.292,\,0.423]$ & 0.995 $[0.972,\,0.999]$ \\
35 & 1.4 & 1.000 $[0.981,\,1.000]$ & 0.620 $[0.551,\,0.684]$ & 1.000 $[0.981,\,1.000]$ \\
35 & 1.6 & 1.000 $[0.981,\,1.000]$ & 0.885 $[0.833,\,0.922]$ & 1.000 $[0.981,\,1.000]$ \\
35 & 1.8 & 1.000 $[0.981,\,1.000]$ & 0.950 $[0.910,\,0.973]$ & 1.000 $[0.981,\,1.000]$ \\
35 & 2.0 & 1.000 $[0.981,\,1.000]$ & 0.985 $[0.957,\,0.995]$ & 1.000 $[0.981,\,1.000]$ \\
40 & 0.0 & 0.070 $[0.042,\,0.114]$ & 0.125 $[0.086,\,0.178]$ & 0.080 $[0.050,\,0.126]$ \\
40 & 0.2 & 0.120 $[0.082,\,0.172]$ & 0.075 $[0.046,\,0.120]$ & 0.120 $[0.082,\,0.172]$ \\
40 & 0.4 & 0.295 $[0.236,\,0.362]$ & 0.115 $[0.078,\,0.167]$ & 0.270 $[0.213,\,0.335]$ \\
40 & 0.6 & 0.610 $[0.541,\,0.675]$ & 0.105 $[0.070,\,0.155]$ & 0.565 $[0.496,\,0.632]$ \\
40 & 0.8 & 0.825 $[0.766,\,0.871]$ & 0.085 $[0.054,\,0.132]$ & 0.825 $[0.766,\,0.871]$ \\
40 & 1.0 & 0.905 $[0.856,\,0.938]$ & 0.200 $[0.150,\,0.261]$ & 0.890 $[0.839,\,0.926]$ \\
40 & 1.2 & 0.990 $[0.964,\,0.997]$ & 0.315 $[0.255,\,0.382]$ & 0.985 $[0.957,\,0.995]$ \\
40 & 1.4 & 0.995 $[0.972,\,0.999]$ & 0.490 $[0.422,\,0.559]$ & 0.995 $[0.972,\,0.999]$ \\
40 & 1.6 & 1.000 $[0.981,\,1.000]$ & 0.675 $[0.607,\,0.736]$ & 1.000 $[0.981,\,1.000]$ \\
40 & 1.8 & 1.000 $[0.981,\,1.000]$ & 0.830 $[0.772,\,0.876]$ & 1.000 $[0.981,\,1.000]$ \\
40 & 2.0 & 1.000 $[0.981,\,1.000]$ & 0.955 $[0.917,\,0.976]$ & 1.000 $[0.981,\,1.000]$ \\
\end{longtable}}

\emph{Second study: Type~I error.}\enspace Table~\ref{tab:sim-basepoint-typeI} gives the empirical rejection rate of the three exact permutation tests of the second study, Section~5.2, in each of its $24$ null cells, at the three attainable sizes of $B = 99$ relabelings. The entries that depart from the size and the rerun of two cells are discussed in Appendix~E.

{\footnotesize\setlength{\tabcolsep}{4pt}
\begin{longtable}{cccc ccc ccc ccc}
\caption{Second simulation study: empirical Type~I error of the three permutation tests ($B = 99$, $R = 200$ replicates per cell, $n = 100$) under the null, all groups drawn from $\mathrm{RiemN}(\samFr, s^2\widehat{\Gamma})$, at the nominal levels $0.01$, $0.05$ and $0.10$, which are the attainable sizes at $B = 99$. The Monte-Carlo standard errors of a proportion at those sizes over $R = 200$ replicates are $0.0070$, $0.0154$ and $0.0212$; entries more than two standard errors from the size are in bold. The two cells at $p = 5$, $g = 3$, balanced, $s = 1$ and $s = 2$ were rerun with $600$ fresh replicates each; see Appendix~E.}\label{tab:sim-basepoint-typeI}\\
\toprule
$p$ & $s$ & $g$ & sizes & \multicolumn{3}{c}{Riemannian ANOVA} & \multicolumn{3}{c}{Tangent MANOVA, $\hat{C}$} & \multicolumn{3}{c}{Tangent MANOVA, $\mathbf{I}$} \\
\cmidrule(lr){5-7}\cmidrule(lr){8-10}\cmidrule(lr){11-13}
 & & & & $0.01$ & $0.05$ & $0.10$ & $0.01$ & $0.05$ & $0.10$ & $0.01$ & $0.05$ & $0.10$ \\
\midrule
\endfirsthead
\multicolumn{13}{l}{\small\itshape Table~\ref{tab:sim-basepoint-typeI} continued}\\
\toprule
$p$ & $s$ & $g$ & sizes & \multicolumn{3}{c}{Riemannian ANOVA} & \multicolumn{3}{c}{Tangent MANOVA, $\hat{C}$} & \multicolumn{3}{c}{Tangent MANOVA, $\mathbf{I}$} \\
\cmidrule(lr){5-7}\cmidrule(lr){8-10}\cmidrule(lr){11-13}
 & & & & $0.01$ & $0.05$ & $0.10$ & $0.01$ & $0.05$ & $0.10$ & $0.01$ & $0.05$ & $0.10$ \\
\midrule
\endhead
\midrule \multicolumn{13}{r}{\small\itshape continued on next page}\\
\endfoot
\bottomrule
\endlastfoot
$5$ & $1$ & $2$ & 50/50 & $0.000$ & $0.035$ & $0.090$ & $0.000$ & $0.045$ & $0.090$ & $0.000$ & $0.045$ & $0.095$ \\
$5$ & $1$ & $2$ & 20/80 & $0.005$ & $0.025$ & $0.065$ & $0.005$ & $0.035$ & $\mathbf{0.055}$ & $0.005$ & $0.025$ & $\mathbf{0.050}$ \\
$5$ & $1$ & $3$ & 34/33/33 & $0.010$ & $0.050$ & $\mathbf{0.155}$ & $0.010$ & $0.070$ & $\mathbf{0.150}$ & $0.005$ & $0.080$ & $0.135$ \\
$5$ & $1$ & $3$ & 15/15/70 & $0.000$ & $0.055$ & $0.100$ & $0.005$ & $0.055$ & $0.100$ & $0.005$ & $0.055$ & $0.100$ \\
$5$ & $2$ & $2$ & 50/50 & $0.010$ & $0.050$ & $0.120$ & $0.005$ & $0.055$ & $0.100$ & $0.005$ & $0.055$ & $0.100$ \\
$5$ & $2$ & $2$ & 20/80 & $0.010$ & $0.040$ & $0.105$ & $0.005$ & $0.045$ & $0.105$ & $0.000$ & $0.060$ & $0.095$ \\
$5$ & $2$ & $3$ & 34/33/33 & $0.015$ & $0.075$ & $\mathbf{0.145}$ & $0.015$ & $\mathbf{0.095}$ & $\mathbf{0.155}$ & $0.015$ & $0.075$ & $\mathbf{0.155}$ \\
$5$ & $2$ & $3$ & 15/15/70 & $0.010$ & $0.025$ & $0.085$ & $0.015$ & $0.035$ & $0.080$ & $0.015$ & $0.040$ & $0.080$ \\
$5$ & $4$ & $2$ & 50/50 & $0.010$ & $0.040$ & $0.075$ & $0.010$ & $0.055$ & $0.110$ & $0.015$ & $0.065$ & $0.085$ \\
$5$ & $4$ & $2$ & 20/80 & $0.000$ & $0.035$ & $0.075$ & $0.005$ & $\mathbf{0.015}$ & $0.065$ & $0.005$ & $0.020$ & $\mathbf{0.055}$ \\
$5$ & $4$ & $3$ & 34/33/33 & $0.015$ & $0.075$ & $0.110$ & $0.015$ & $0.055$ & $0.100$ & $0.015$ & $0.060$ & $0.115$ \\
$5$ & $4$ & $3$ & 15/15/70 & $0.015$ & $0.040$ & $0.095$ & $0.015$ & $0.050$ & $0.065$ & $0.020$ & $0.060$ & $0.090$ \\
$10$ & $1$ & $2$ & 50/50 & $0.010$ & $0.050$ & $0.100$ & $0.010$ & $0.065$ & $0.100$ & $0.010$ & $0.040$ & $0.120$ \\
$10$ & $1$ & $2$ & 20/80 & $0.005$ & $0.045$ & $0.095$ & $0.010$ & $0.035$ & $0.100$ & $0.010$ & $0.045$ & $0.095$ \\
$10$ & $1$ & $3$ & 34/33/33 & $0.010$ & $0.025$ & $0.070$ & $0.005$ & $0.035$ & $0.070$ & $0.000$ & $0.030$ & $0.070$ \\
$10$ & $1$ & $3$ & 15/15/70 & $0.010$ & $0.065$ & $0.100$ & $0.020$ & $0.080$ & $0.125$ & $0.015$ & $0.080$ & $0.115$ \\
$10$ & $2$ & $2$ & 50/50 & $0.005$ & $0.040$ & $0.065$ & $0.005$ & $0.040$ & $0.070$ & $0.005$ & $0.025$ & $0.065$ \\
$10$ & $2$ & $2$ & 20/80 & $0.000$ & $0.050$ & $0.075$ & $0.005$ & $0.045$ & $0.075$ & $0.000$ & $0.045$ & $0.105$ \\
$10$ & $2$ & $3$ & 34/33/33 & $0.010$ & $0.050$ & $0.105$ & $0.010$ & $0.055$ & $0.120$ & $0.010$ & $0.065$ & $0.135$ \\
$10$ & $2$ & $3$ & 15/15/70 & $0.015$ & $0.060$ & $0.115$ & $0.000$ & $0.075$ & $0.125$ & $0.010$ & $0.040$ & $0.115$ \\
$10$ & $4$ & $2$ & 50/50 & $0.000$ & $0.050$ & $0.110$ & $0.015$ & $0.045$ & $0.105$ & $0.015$ & $0.045$ & $0.100$ \\
$10$ & $4$ & $2$ & 20/80 & $0.000$ & $0.025$ & $0.065$ & $0.000$ & $\mathbf{0.015}$ & $0.065$ & $0.000$ & $0.020$ & $0.075$ \\
$10$ & $4$ & $3$ & 34/33/33 & $0.005$ & $0.025$ & $0.105$ & $0.005$ & $0.040$ & $0.090$ & $0.005$ & $0.030$ & $0.070$ \\
$10$ & $4$ & $3$ & 15/15/70 & $0.010$ & $0.035$ & $0.115$ & $0.015$ & $0.050$ & $0.120$ & $0.010$ & $0.040$ & $0.125$ \\
\end{longtable}}

\emph{Second study: power.}\enspace Table~\ref{tab:sim-basepoint-power} lists the empirical power of the three tests in each of the $144$ power cells of the second study, with the gain of $\Lambda^\ast$ over each comparator and its $95\%$ interval, the mean difference of the two rejection indicators over the $R = 200$ replicates plus or minus $1.96$ times its standard error; Table~2 is the mean by $(p, s)$. In cells where all three tests reject in every replicate the difference and its interval are zero.

{\footnotesize\setlength{\tabcolsep}{2.2pt}
\begin{longtable}{ccccc c ccc cc}
\caption{Second simulation study: empirical power at the attainable size $0.06$ ($B = 49$, $R = 200$, $n = 100$) of the Riemannian ANOVA ($\Lambda^\ast$), tangent-space MANOVA at the pooled Fr\'echet mean ($\hat{C}$) and tangent-space MANOVA at the identity ($\mathbf{I}$) in every power cell, with the gain of $\Lambda^\ast$ over each comparator and its $95\%$ confidence interval for the paired difference. Placement: displaced, one group (the first for $g = 2$, the third for $g = 3$) at distance $\theta$ from the center; triangle, the three group means at distance $\theta$ from the center along directions $120^\circ$ apart. The Wilson $95\%$ interval of each power has half-width at most $0.07$.}\label{tab:sim-basepoint-power}\\
\toprule
$p$ & $s$ & $g$ & sizes & placement & $\theta$ & $\Lambda^\ast$ & $\hat{C}$ & $\mathbf{I}$ & gain over $\hat{C}$ (95\% CI) & gain over $\mathbf{I}$ (95\% CI) \\
\midrule
\endfirsthead
\multicolumn{11}{l}{\small\itshape Table~\ref{tab:sim-basepoint-power} continued}\\
\toprule
$p$ & $s$ & $g$ & sizes & placement & $\theta$ & $\Lambda^\ast$ & $\hat{C}$ & $\mathbf{I}$ & gain over $\hat{C}$ (95\% CI) & gain over $\mathbf{I}$ (95\% CI) \\
\midrule
\endhead
\midrule \multicolumn{11}{r}{\small\itshape continued on next page}\\
\endfoot
\bottomrule
\endlastfoot
$5$ & $1$ & $2$ & 50/50 & displaced & $0.32$ & $0.210$ & $0.225$ & $0.230$ & $-0.015$ $(-0.053,\,0.023)$ & $-0.020$ $(-0.066,\,0.026)$ \\
$5$ & $1$ & $2$ & 50/50 & displaced & $0.46$ & $0.450$ & $0.435$ & $0.425$ & $0.015$ $(-0.034,\,0.064)$ & $0.025$ $(-0.022,\,0.072)$ \\
$5$ & $1$ & $2$ & 50/50 & displaced & $0.66$ & $0.785$ & $0.795$ & $0.800$ & $-0.010$ $(-0.044,\,0.024)$ & $-0.015$ $(-0.048,\,0.018)$ \\
$5$ & $1$ & $2$ & 50/50 & displaced & $0.87$ & $0.985$ & $0.960$ & $0.965$ & $0.025$ $(0.003,\,0.047)$ & $0.020$ $(0.001,\,0.039)$ \\
$5$ & $1$ & $2$ & 20/80 & displaced & $0.32$ & $0.205$ & $0.245$ & $0.245$ & $-0.040$ $(-0.081,\,0.001)$ & $-0.040$ $(-0.081,\,0.001)$ \\
$5$ & $1$ & $2$ & 20/80 & displaced & $0.46$ & $0.325$ & $0.290$ & $0.270$ & $0.035$ $(-0.014,\,0.084)$ & $0.055$ $(0.003,\,0.107)$ \\
$5$ & $1$ & $2$ & 20/80 & displaced & $0.66$ & $0.625$ & $0.580$ & $0.585$ & $0.045$ $(-0.014,\,0.104)$ & $0.040$ $(-0.019,\,0.099)$ \\
$5$ & $1$ & $2$ & 20/80 & displaced & $0.87$ & $0.815$ & $0.805$ & $0.805$ & $0.010$ $(-0.034,\,0.054)$ & $0.010$ $(-0.032,\,0.052)$ \\
$5$ & $1$ & $3$ & 34/33/33 & displaced & $0.32$ & $0.150$ & $0.170$ & $0.160$ & $-0.020$ $(-0.057,\,0.017)$ & $-0.010$ $(-0.049,\,0.029)$ \\
$5$ & $1$ & $3$ & 34/33/33 & displaced & $0.46$ & $0.310$ & $0.235$ & $0.270$ & $0.075$ $(0.027,\,0.123)$ & $0.040$ $(-0.008,\,0.088)$ \\
$5$ & $1$ & $3$ & 34/33/33 & displaced & $0.66$ & $0.525$ & $0.500$ & $0.495$ & $0.025$ $(-0.020,\,0.070)$ & $0.030$ $(-0.018,\,0.078)$ \\
$5$ & $1$ & $3$ & 34/33/33 & displaced & $0.87$ & $0.875$ & $0.855$ & $0.860$ & $0.020$ $(-0.019,\,0.059)$ & $0.015$ $(-0.020,\,0.050)$ \\
$5$ & $1$ & $3$ & 34/33/33 & triangle & $0.32$ & $0.760$ & $0.760$ & $0.765$ & $0.000$ $(-0.044,\,0.044)$ & $-0.005$ $(-0.048,\,0.038)$ \\
$5$ & $1$ & $3$ & 34/33/33 & triangle & $0.46$ & $0.990$ & $0.995$ & $0.995$ & $-0.005$ $(-0.015,\,0.005)$ & $-0.005$ $(-0.015,\,0.005)$ \\
$5$ & $1$ & $3$ & 34/33/33 & triangle & $0.66$ & $1.000$ & $1.000$ & $1.000$ & $0.000$ $(0.000,\,0.000)$ & $0.000$ $(0.000,\,0.000)$ \\
$5$ & $1$ & $3$ & 34/33/33 & triangle & $0.87$ & $1.000$ & $1.000$ & $1.000$ & $0.000$ $(0.000,\,0.000)$ & $0.000$ $(0.000,\,0.000)$ \\
$5$ & $1$ & $3$ & 15/15/70 & displaced & $0.32$ & $0.105$ & $0.145$ & $0.145$ & $-0.040$ $(-0.076,\,-0.004)$ & $-0.040$ $(-0.076,\,-0.004)$ \\
$5$ & $1$ & $3$ & 15/15/70 & displaced & $0.46$ & $0.250$ & $0.230$ & $0.230$ & $0.020$ $(-0.019,\,0.059)$ & $0.020$ $(-0.019,\,0.059)$ \\
$5$ & $1$ & $3$ & 15/15/70 & displaced & $0.66$ & $0.485$ & $0.485$ & $0.495$ & $0.000$ $(-0.048,\,0.048)$ & $-0.010$ $(-0.056,\,0.036)$ \\
$5$ & $1$ & $3$ & 15/15/70 & displaced & $0.87$ & $0.835$ & $0.800$ & $0.805$ & $0.035$ $(-0.012,\,0.082)$ & $0.030$ $(-0.016,\,0.076)$ \\
$5$ & $1$ & $3$ & 15/15/70 & triangle & $0.32$ & $0.540$ & $0.485$ & $0.510$ & $0.055$ $(0.009,\,0.101)$ & $0.030$ $(-0.020,\,0.080)$ \\
$5$ & $1$ & $3$ & 15/15/70 & triangle & $0.46$ & $0.930$ & $0.930$ & $0.925$ & $0.000$ $(-0.024,\,0.024)$ & $0.005$ $(-0.012,\,0.022)$ \\
$5$ & $1$ & $3$ & 15/15/70 & triangle & $0.66$ & $1.000$ & $1.000$ & $1.000$ & $0.000$ $(0.000,\,0.000)$ & $0.000$ $(0.000,\,0.000)$ \\
$5$ & $1$ & $3$ & 15/15/70 & triangle & $0.87$ & $1.000$ & $1.000$ & $1.000$ & $0.000$ $(0.000,\,0.000)$ & $0.000$ $(0.000,\,0.000)$ \\
$5$ & $2$ & $2$ & 50/50 & displaced & $0.50$ & $0.160$ & $0.150$ & $0.160$ & $0.010$ $(-0.032,\,0.052)$ & $0.000$ $(-0.039,\,0.039)$ \\
$5$ & $2$ & $2$ & 50/50 & displaced & $0.71$ & $0.350$ & $0.285$ & $0.290$ & $0.065$ $(0.013,\,0.117)$ & $0.060$ $(0.009,\,0.111)$ \\
$5$ & $2$ & $2$ & 50/50 & displaced & $1.00$ & $0.570$ & $0.510$ & $0.540$ & $0.060$ $(0.013,\,0.107)$ & $0.030$ $(-0.020,\,0.080)$ \\
$5$ & $2$ & $2$ & 50/50 & displaced & $1.59$ & $0.985$ & $0.980$ & $0.980$ & $0.005$ $(-0.017,\,0.027)$ & $0.005$ $(-0.017,\,0.027)$ \\
$5$ & $2$ & $2$ & 20/80 & displaced & $0.50$ & $0.125$ & $0.120$ & $0.120$ & $0.005$ $(-0.021,\,0.031)$ & $0.005$ $(-0.021,\,0.031)$ \\
$5$ & $2$ & $2$ & 20/80 & displaced & $0.71$ & $0.235$ & $0.195$ & $0.195$ & $0.040$ $(-0.001,\,0.081)$ & $0.040$ $(-0.001,\,0.081)$ \\
$5$ & $2$ & $2$ & 20/80 & displaced & $1.00$ & $0.430$ & $0.385$ & $0.380$ & $0.045$ $(-0.006,\,0.096)$ & $0.050$ $(-0.003,\,0.103)$ \\
$5$ & $2$ & $2$ & 20/80 & displaced & $1.59$ & $0.860$ & $0.825$ & $0.825$ & $0.035$ $(-0.000,\,0.070)$ & $0.035$ $(-0.003,\,0.073)$ \\
$5$ & $2$ & $3$ & 34/33/33 & displaced & $0.50$ & $0.145$ & $0.125$ & $0.120$ & $0.020$ $(-0.022,\,0.062)$ & $0.025$ $(-0.020,\,0.070)$ \\
$5$ & $2$ & $3$ & 34/33/33 & displaced & $0.71$ & $0.205$ & $0.235$ & $0.225$ & $-0.030$ $(-0.074,\,0.014)$ & $-0.020$ $(-0.064,\,0.024)$ \\
$5$ & $2$ & $3$ & 34/33/33 & displaced & $1.00$ & $0.415$ & $0.360$ & $0.370$ & $0.055$ $(0.011,\,0.099)$ & $0.045$ $(0.000,\,0.090)$ \\
$5$ & $2$ & $3$ & 34/33/33 & displaced & $1.59$ & $0.890$ & $0.835$ & $0.820$ & $0.055$ $(0.013,\,0.097)$ & $0.070$ $(0.027,\,0.113)$ \\
$5$ & $2$ & $3$ & 34/33/33 & triangle & $0.50$ & $0.685$ & $0.620$ & $0.600$ & $0.065$ $(0.015,\,0.115)$ & $0.085$ $(0.030,\,0.140)$ \\
$5$ & $2$ & $3$ & 34/33/33 & triangle & $0.71$ & $0.945$ & $0.915$ & $0.895$ & $0.030$ $(-0.001,\,0.061)$ & $0.050$ $(0.014,\,0.086)$ \\
$5$ & $2$ & $3$ & 34/33/33 & triangle & $1.00$ & $1.000$ & $1.000$ & $1.000$ & $0.000$ $(0.000,\,0.000)$ & $0.000$ $(0.000,\,0.000)$ \\
$5$ & $2$ & $3$ & 34/33/33 & triangle & $1.59$ & $1.000$ & $1.000$ & $1.000$ & $0.000$ $(0.000,\,0.000)$ & $0.000$ $(0.000,\,0.000)$ \\
$5$ & $2$ & $3$ & 15/15/70 & displaced & $0.50$ & $0.120$ & $0.110$ & $0.110$ & $0.010$ $(-0.024,\,0.044)$ & $0.010$ $(-0.027,\,0.047)$ \\
$5$ & $2$ & $3$ & 15/15/70 & displaced & $0.71$ & $0.185$ & $0.185$ & $0.180$ & $0.000$ $(-0.039,\,0.039)$ & $0.005$ $(-0.036,\,0.046)$ \\
$5$ & $2$ & $3$ & 15/15/70 & displaced & $1.00$ & $0.425$ & $0.350$ & $0.360$ & $0.075$ $(0.023,\,0.127)$ & $0.065$ $(0.015,\,0.115)$ \\
$5$ & $2$ & $3$ & 15/15/70 & displaced & $1.59$ & $0.855$ & $0.795$ & $0.815$ & $0.060$ $(0.017,\,0.103)$ & $0.040$ $(0.001,\,0.079)$ \\
$5$ & $2$ & $3$ & 15/15/70 & triangle & $0.50$ & $0.490$ & $0.400$ & $0.415$ & $0.090$ $(0.041,\,0.139)$ & $0.075$ $(0.027,\,0.123)$ \\
$5$ & $2$ & $3$ & 15/15/70 & triangle & $0.71$ & $0.835$ & $0.770$ & $0.765$ & $0.065$ $(0.023,\,0.107)$ & $0.070$ $(0.025,\,0.115)$ \\
$5$ & $2$ & $3$ & 15/15/70 & triangle & $1.00$ & $1.000$ & $0.995$ & $0.995$ & $0.005$ $(-0.005,\,0.015)$ & $0.005$ $(-0.005,\,0.015)$ \\
$5$ & $2$ & $3$ & 15/15/70 & triangle & $1.59$ & $1.000$ & $1.000$ & $1.000$ & $0.000$ $(0.000,\,0.000)$ & $0.000$ $(0.000,\,0.000)$ \\
$5$ & $4$ & $2$ & 50/50 & displaced & $1.05$ & $0.320$ & $0.300$ & $0.310$ & $0.020$ $(-0.034,\,0.074)$ & $0.010$ $(-0.040,\,0.060)$ \\
$5$ & $4$ & $2$ & 50/50 & displaced & $1.35$ & $0.520$ & $0.445$ & $0.455$ & $0.075$ $(0.013,\,0.137)$ & $0.065$ $(0.003,\,0.127)$ \\
$5$ & $4$ & $2$ & 50/50 & displaced & $1.72$ & $0.695$ & $0.715$ & $0.685$ & $-0.020$ $(-0.082,\,0.042)$ & $0.010$ $(-0.049,\,0.069)$ \\
$5$ & $4$ & $2$ & 50/50 & displaced & $2.52$ & $0.980$ & $0.980$ & $0.975$ & $0.000$ $(-0.020,\,0.020)$ & $0.005$ $(-0.017,\,0.027)$ \\
$5$ & $4$ & $2$ & 20/80 & displaced & $1.05$ & $0.175$ & $0.200$ & $0.180$ & $-0.025$ $(-0.085,\,0.035)$ & $-0.005$ $(-0.060,\,0.050)$ \\
$5$ & $4$ & $2$ & 20/80 & displaced & $1.35$ & $0.310$ & $0.285$ & $0.315$ & $0.025$ $(-0.030,\,0.080)$ & $-0.005$ $(-0.056,\,0.046)$ \\
$5$ & $4$ & $2$ & 20/80 & displaced & $1.72$ & $0.435$ & $0.400$ & $0.405$ & $0.035$ $(-0.026,\,0.096)$ & $0.030$ $(-0.032,\,0.092)$ \\
$5$ & $4$ & $2$ & 20/80 & displaced & $2.52$ & $0.830$ & $0.840$ & $0.820$ & $-0.010$ $(-0.054,\,0.034)$ & $0.010$ $(-0.038,\,0.058)$ \\
$5$ & $4$ & $3$ & 34/33/33 & displaced & $1.05$ & $0.170$ & $0.170$ & $0.190$ & $0.000$ $(-0.048,\,0.048)$ & $-0.020$ $(-0.070,\,0.030)$ \\
$5$ & $4$ & $3$ & 34/33/33 & displaced & $1.35$ & $0.230$ & $0.200$ & $0.205$ & $0.030$ $(-0.029,\,0.089)$ & $0.025$ $(-0.031,\,0.081)$ \\
$5$ & $4$ & $3$ & 34/33/33 & displaced & $1.72$ & $0.450$ & $0.475$ & $0.470$ & $-0.025$ $(-0.089,\,0.039)$ & $-0.020$ $(-0.085,\,0.045)$ \\
$5$ & $4$ & $3$ & 34/33/33 & displaced & $2.52$ & $0.870$ & $0.880$ & $0.865$ & $-0.010$ $(-0.058,\,0.038)$ & $0.005$ $(-0.044,\,0.054)$ \\
$5$ & $4$ & $3$ & 34/33/33 & triangle & $1.05$ & $0.880$ & $0.900$ & $0.885$ & $-0.020$ $(-0.054,\,0.014)$ & $-0.005$ $(-0.040,\,0.030)$ \\
$5$ & $4$ & $3$ & 34/33/33 & triangle & $1.35$ & $0.995$ & $0.990$ & $0.995$ & $0.005$ $(-0.005,\,0.015)$ & $0.000$ $(0.000,\,0.000)$ \\
$5$ & $4$ & $3$ & 34/33/33 & triangle & $1.72$ & $1.000$ & $1.000$ & $1.000$ & $0.000$ $(0.000,\,0.000)$ & $0.000$ $(0.000,\,0.000)$ \\
$5$ & $4$ & $3$ & 34/33/33 & triangle & $2.52$ & $1.000$ & $1.000$ & $1.000$ & $0.000$ $(0.000,\,0.000)$ & $0.000$ $(0.000,\,0.000)$ \\
$5$ & $4$ & $3$ & 15/15/70 & displaced & $1.05$ & $0.170$ & $0.175$ & $0.190$ & $-0.005$ $(-0.054,\,0.044)$ & $-0.020$ $(-0.072,\,0.032)$ \\
$5$ & $4$ & $3$ & 15/15/70 & displaced & $1.35$ & $0.260$ & $0.290$ & $0.290$ & $-0.030$ $(-0.089,\,0.029)$ & $-0.030$ $(-0.094,\,0.034)$ \\
$5$ & $4$ & $3$ & 15/15/70 & displaced & $1.72$ & $0.455$ & $0.440$ & $0.465$ & $0.015$ $(-0.046,\,0.076)$ & $-0.010$ $(-0.067,\,0.047)$ \\
$5$ & $4$ & $3$ & 15/15/70 & displaced & $2.52$ & $0.815$ & $0.840$ & $0.825$ & $-0.025$ $(-0.076,\,0.026)$ & $-0.010$ $(-0.060,\,0.040)$ \\
$5$ & $4$ & $3$ & 15/15/70 & triangle & $1.05$ & $0.755$ & $0.745$ & $0.745$ & $0.010$ $(-0.040,\,0.060)$ & $0.010$ $(-0.038,\,0.058)$ \\
$5$ & $4$ & $3$ & 15/15/70 & triangle & $1.35$ & $0.965$ & $0.975$ & $0.960$ & $-0.010$ $(-0.034,\,0.014)$ & $0.005$ $(-0.024,\,0.034)$ \\
$5$ & $4$ & $3$ & 15/15/70 & triangle & $1.72$ & $1.000$ & $1.000$ & $1.000$ & $0.000$ $(0.000,\,0.000)$ & $0.000$ $(0.000,\,0.000)$ \\
$5$ & $4$ & $3$ & 15/15/70 & triangle & $2.52$ & $1.000$ & $1.000$ & $1.000$ & $0.000$ $(0.000,\,0.000)$ & $0.000$ $(0.000,\,0.000)$ \\
$10$ & $1$ & $2$ & 50/50 & displaced & $0.57$ & $0.155$ & $0.165$ & $0.160$ & $-0.010$ $(-0.047,\,0.027)$ & $-0.005$ $(-0.043,\,0.033)$ \\
$10$ & $1$ & $2$ & 50/50 & displaced & $1.06$ & $0.540$ & $0.470$ & $0.490$ & $0.070$ $(0.021,\,0.119)$ & $0.050$ $(-0.005,\,0.105)$ \\
$10$ & $1$ & $2$ & 50/50 & displaced & $1.41$ & $0.815$ & $0.785$ & $0.780$ & $0.030$ $(-0.024,\,0.084)$ & $0.035$ $(-0.019,\,0.089)$ \\
$10$ & $1$ & $2$ & 50/50 & displaced & $1.78$ & $0.980$ & $0.975$ & $0.980$ & $0.005$ $(-0.012,\,0.022)$ & $0.000$ $(-0.014,\,0.014)$ \\
$10$ & $1$ & $2$ & 20/80 & displaced & $0.57$ & $0.175$ & $0.160$ & $0.145$ & $0.015$ $(-0.032,\,0.062)$ & $0.030$ $(-0.014,\,0.074)$ \\
$10$ & $1$ & $2$ & 20/80 & displaced & $1.06$ & $0.365$ & $0.280$ & $0.310$ & $0.085$ $(0.033,\,0.137)$ & $0.055$ $(0.001,\,0.109)$ \\
$10$ & $1$ & $2$ & 20/80 & displaced & $1.41$ & $0.650$ & $0.525$ & $0.540$ & $0.125$ $(0.073,\,0.177)$ & $0.110$ $(0.058,\,0.162)$ \\
$10$ & $1$ & $2$ & 20/80 & displaced & $1.78$ & $0.865$ & $0.780$ & $0.785$ & $0.085$ $(0.039,\,0.131)$ & $0.080$ $(0.035,\,0.125)$ \\
$10$ & $1$ & $3$ & 34/33/33 & displaced & $0.57$ & $0.080$ & $0.045$ & $0.040$ & $0.035$ $(0.006,\,0.064)$ & $0.040$ $(0.009,\,0.071)$ \\
$10$ & $1$ & $3$ & 34/33/33 & displaced & $1.06$ & $0.355$ & $0.315$ & $0.305$ & $0.040$ $(-0.006,\,0.086)$ & $0.050$ $(0.004,\,0.096)$ \\
$10$ & $1$ & $3$ & 34/33/33 & displaced & $1.41$ & $0.595$ & $0.520$ & $0.505$ & $0.075$ $(0.013,\,0.137)$ & $0.090$ $(0.031,\,0.149)$ \\
$10$ & $1$ & $3$ & 34/33/33 & displaced & $1.78$ & $0.830$ & $0.760$ & $0.755$ & $0.070$ $(0.023,\,0.117)$ & $0.075$ $(0.027,\,0.123)$ \\
$10$ & $1$ & $3$ & 34/33/33 & triangle & $0.57$ & $0.775$ & $0.740$ & $0.755$ & $0.035$ $(-0.010,\,0.080)$ & $0.020$ $(-0.024,\,0.064)$ \\
$10$ & $1$ & $3$ & 34/33/33 & triangle & $1.06$ & $1.000$ & $1.000$ & $1.000$ & $0.000$ $(0.000,\,0.000)$ & $0.000$ $(0.000,\,0.000)$ \\
$10$ & $1$ & $3$ & 34/33/33 & triangle & $1.41$ & $1.000$ & $1.000$ & $1.000$ & $0.000$ $(0.000,\,0.000)$ & $0.000$ $(0.000,\,0.000)$ \\
$10$ & $1$ & $3$ & 34/33/33 & triangle & $1.78$ & $1.000$ & $1.000$ & $1.000$ & $0.000$ $(0.000,\,0.000)$ & $0.000$ $(0.000,\,0.000)$ \\
$10$ & $1$ & $3$ & 15/15/70 & displaced & $0.57$ & $0.130$ & $0.145$ & $0.140$ & $-0.015$ $(-0.044,\,0.014)$ & $-0.010$ $(-0.041,\,0.021)$ \\
$10$ & $1$ & $3$ & 15/15/70 & displaced & $1.06$ & $0.320$ & $0.275$ & $0.280$ & $0.045$ $(0.000,\,0.090)$ & $0.040$ $(-0.010,\,0.090)$ \\
$10$ & $1$ & $3$ & 15/15/70 & displaced & $1.41$ & $0.575$ & $0.535$ & $0.495$ & $0.040$ $(-0.010,\,0.090)$ & $0.080$ $(0.031,\,0.129)$ \\
$10$ & $1$ & $3$ & 15/15/70 & displaced & $1.78$ & $0.805$ & $0.775$ & $0.750$ & $0.030$ $(-0.018,\,0.078)$ & $0.055$ $(0.003,\,0.107)$ \\
$10$ & $1$ & $3$ & 15/15/70 & triangle & $0.57$ & $0.560$ & $0.515$ & $0.510$ & $0.045$ $(0.005,\,0.085)$ & $0.050$ $(0.004,\,0.096)$ \\
$10$ & $1$ & $3$ & 15/15/70 & triangle & $1.06$ & $1.000$ & $1.000$ & $1.000$ & $0.000$ $(0.000,\,0.000)$ & $0.000$ $(0.000,\,0.000)$ \\
$10$ & $1$ & $3$ & 15/15/70 & triangle & $1.41$ & $1.000$ & $1.000$ & $1.000$ & $0.000$ $(0.000,\,0.000)$ & $0.000$ $(0.000,\,0.000)$ \\
$10$ & $1$ & $3$ & 15/15/70 & triangle & $1.78$ & $1.000$ & $1.000$ & $1.000$ & $0.000$ $(0.000,\,0.000)$ & $0.000$ $(0.000,\,0.000)$ \\
$10$ & $2$ & $2$ & 50/50 & displaced & $0.79$ & $0.120$ & $0.100$ & $0.105$ & $0.020$ $(-0.017,\,0.057)$ & $0.015$ $(-0.025,\,0.055)$ \\
$10$ & $2$ & $2$ & 50/50 & displaced & $1.49$ & $0.430$ & $0.350$ & $0.365$ & $0.080$ $(0.024,\,0.136)$ & $0.065$ $(0.008,\,0.122)$ \\
$10$ & $2$ & $2$ & 50/50 & displaced & $2.38$ & $0.895$ & $0.745$ & $0.745$ & $0.150$ $(0.097,\,0.203)$ & $0.150$ $(0.091,\,0.209)$ \\
$10$ & $2$ & $2$ & 50/50 & displaced & $3.36$ & $0.995$ & $0.995$ & $0.995$ & $0.000$ $(0.000,\,0.000)$ & $0.000$ $(0.000,\,0.000)$ \\
$10$ & $2$ & $2$ & 20/80 & displaced & $0.79$ & $0.095$ & $0.120$ & $0.080$ & $-0.025$ $(-0.060,\,0.010)$ & $0.015$ $(-0.020,\,0.050)$ \\
$10$ & $2$ & $2$ & 20/80 & displaced & $1.49$ & $0.295$ & $0.205$ & $0.210$ & $0.090$ $(0.041,\,0.139)$ & $0.085$ $(0.035,\,0.135)$ \\
$10$ & $2$ & $2$ & 20/80 & displaced & $2.38$ & $0.590$ & $0.455$ & $0.475$ & $0.135$ $(0.080,\,0.190)$ & $0.115$ $(0.053,\,0.177)$ \\
$10$ & $2$ & $2$ & 20/80 & displaced & $3.36$ & $0.935$ & $0.855$ & $0.855$ & $0.080$ $(0.035,\,0.125)$ & $0.080$ $(0.037,\,0.123)$ \\
$10$ & $2$ & $3$ & 34/33/33 & displaced & $0.79$ & $0.100$ & $0.070$ & $0.095$ & $0.030$ $(-0.007,\,0.067)$ & $0.005$ $(-0.030,\,0.040)$ \\
$10$ & $2$ & $3$ & 34/33/33 & displaced & $1.49$ & $0.235$ & $0.195$ & $0.175$ & $0.040$ $(-0.001,\,0.081)$ & $0.060$ $(0.015,\,0.105)$ \\
$10$ & $2$ & $3$ & 34/33/33 & displaced & $2.38$ & $0.610$ & $0.445$ & $0.430$ & $0.165$ $(0.110,\,0.220)$ & $0.180$ $(0.123,\,0.237)$ \\
$10$ & $2$ & $3$ & 34/33/33 & displaced & $3.36$ & $0.920$ & $0.780$ & $0.775$ & $0.140$ $(0.090,\,0.190)$ & $0.145$ $(0.094,\,0.196)$ \\
$10$ & $2$ & $3$ & 34/33/33 & triangle & $0.79$ & $0.500$ & $0.430$ & $0.425$ & $0.070$ $(0.021,\,0.119)$ & $0.075$ $(0.021,\,0.129)$ \\
$10$ & $2$ & $3$ & 34/33/33 & triangle & $1.49$ & $0.995$ & $0.995$ & $0.995$ & $0.000$ $(0.000,\,0.000)$ & $0.000$ $(0.000,\,0.000)$ \\
$10$ & $2$ & $3$ & 34/33/33 & triangle & $2.38$ & $1.000$ & $1.000$ & $1.000$ & $0.000$ $(0.000,\,0.000)$ & $0.000$ $(0.000,\,0.000)$ \\
$10$ & $2$ & $3$ & 34/33/33 & triangle & $3.36$ & $1.000$ & $1.000$ & $1.000$ & $0.000$ $(0.000,\,0.000)$ & $0.000$ $(0.000,\,0.000)$ \\
$10$ & $2$ & $3$ & 15/15/70 & displaced & $0.79$ & $0.105$ & $0.085$ & $0.060$ & $0.020$ $(-0.024,\,0.064)$ & $0.045$ $(0.003,\,0.087)$ \\
$10$ & $2$ & $3$ & 15/15/70 & displaced & $1.49$ & $0.260$ & $0.215$ & $0.215$ & $0.045$ $(-0.006,\,0.096)$ & $0.045$ $(-0.008,\,0.098)$ \\
$10$ & $2$ & $3$ & 15/15/70 & displaced & $2.38$ & $0.560$ & $0.395$ & $0.400$ & $0.165$ $(0.106,\,0.224)$ & $0.160$ $(0.099,\,0.221)$ \\
$10$ & $2$ & $3$ & 15/15/70 & displaced & $3.36$ & $0.920$ & $0.805$ & $0.805$ & $0.115$ $(0.065,\,0.165)$ & $0.115$ $(0.069,\,0.161)$ \\
$10$ & $2$ & $3$ & 15/15/70 & triangle & $0.79$ & $0.435$ & $0.375$ & $0.360$ & $0.060$ $(0.009,\,0.111)$ & $0.075$ $(0.020,\,0.130)$ \\
$10$ & $2$ & $3$ & 15/15/70 & triangle & $1.49$ & $0.980$ & $0.965$ & $0.965$ & $0.015$ $(-0.011,\,0.041)$ & $0.015$ $(-0.014,\,0.044)$ \\
$10$ & $2$ & $3$ & 15/15/70 & triangle & $2.38$ & $1.000$ & $1.000$ & $1.000$ & $0.000$ $(0.000,\,0.000)$ & $0.000$ $(0.000,\,0.000)$ \\
$10$ & $2$ & $3$ & 15/15/70 & triangle & $3.36$ & $1.000$ & $1.000$ & $1.000$ & $0.000$ $(0.000,\,0.000)$ & $0.000$ $(0.000,\,0.000)$ \\
$10$ & $4$ & $2$ & 50/50 & displaced & $1.49$ & $0.190$ & $0.145$ & $0.145$ & $0.045$ $(-0.009,\,0.099)$ & $0.045$ $(-0.006,\,0.096)$ \\
$10$ & $4$ & $2$ & 50/50 & displaced & $2.23$ & $0.440$ & $0.335$ & $0.305$ & $0.105$ $(0.045,\,0.165)$ & $0.135$ $(0.072,\,0.198)$ \\
$10$ & $4$ & $2$ & 50/50 & displaced & $2.90$ & $0.620$ & $0.520$ & $0.535$ & $0.100$ $(0.028,\,0.172)$ & $0.085$ $(0.013,\,0.157)$ \\
$10$ & $4$ & $2$ & 50/50 & displaced & $3.60$ & $0.915$ & $0.785$ & $0.775$ & $0.130$ $(0.074,\,0.186)$ & $0.140$ $(0.083,\,0.197)$ \\
$10$ & $4$ & $2$ & 20/80 & displaced & $1.49$ & $0.140$ & $0.115$ & $0.120$ & $0.025$ $(-0.018,\,0.068)$ & $0.020$ $(-0.022,\,0.062)$ \\
$10$ & $4$ & $2$ & 20/80 & displaced & $2.23$ & $0.235$ & $0.220$ & $0.220$ & $0.015$ $(-0.040,\,0.070)$ & $0.015$ $(-0.043,\,0.073)$ \\
$10$ & $4$ & $2$ & 20/80 & displaced & $2.90$ & $0.435$ & $0.315$ & $0.325$ & $0.120$ $(0.055,\,0.185)$ & $0.110$ $(0.042,\,0.178)$ \\
$10$ & $4$ & $2$ & 20/80 & displaced & $3.60$ & $0.615$ & $0.540$ & $0.570$ & $0.075$ $(0.010,\,0.140)$ & $0.045$ $(-0.022,\,0.112)$ \\
$10$ & $4$ & $3$ & 34/33/33 & displaced & $1.49$ & $0.145$ & $0.110$ & $0.155$ & $0.035$ $(-0.012,\,0.082)$ & $-0.010$ $(-0.064,\,0.044)$ \\
$10$ & $4$ & $3$ & 34/33/33 & displaced & $2.23$ & $0.200$ & $0.155$ & $0.175$ & $0.045$ $(-0.016,\,0.106)$ & $0.025$ $(-0.031,\,0.081)$ \\
$10$ & $4$ & $3$ & 34/33/33 & displaced & $2.90$ & $0.450$ & $0.385$ & $0.370$ & $0.065$ $(0.006,\,0.124)$ & $0.080$ $(0.013,\,0.147)$ \\
$10$ & $4$ & $3$ & 34/33/33 & displaced & $3.60$ & $0.655$ & $0.460$ & $0.490$ & $0.195$ $(0.133,\,0.257)$ & $0.165$ $(0.100,\,0.230)$ \\
$10$ & $4$ & $3$ & 34/33/33 & triangle & $1.49$ & $0.860$ & $0.805$ & $0.840$ & $0.055$ $(0.001,\,0.109)$ & $0.020$ $(-0.032,\,0.072)$ \\
$10$ & $4$ & $3$ & 34/33/33 & triangle & $2.23$ & $1.000$ & $1.000$ & $1.000$ & $0.000$ $(0.000,\,0.000)$ & $0.000$ $(0.000,\,0.000)$ \\
$10$ & $4$ & $3$ & 34/33/33 & triangle & $2.90$ & $1.000$ & $1.000$ & $1.000$ & $0.000$ $(0.000,\,0.000)$ & $0.000$ $(0.000,\,0.000)$ \\
$10$ & $4$ & $3$ & 34/33/33 & triangle & $3.60$ & $1.000$ & $1.000$ & $1.000$ & $0.000$ $(0.000,\,0.000)$ & $0.000$ $(0.000,\,0.000)$ \\
$10$ & $4$ & $3$ & 15/15/70 & displaced & $1.49$ & $0.115$ & $0.115$ & $0.110$ & $0.000$ $(-0.039,\,0.039)$ & $0.005$ $(-0.038,\,0.048)$ \\
$10$ & $4$ & $3$ & 15/15/70 & displaced & $2.23$ & $0.230$ & $0.235$ & $0.175$ & $-0.005$ $(-0.060,\,0.050)$ & $0.055$ $(-0.001,\,0.111)$ \\
$10$ & $4$ & $3$ & 15/15/70 & displaced & $2.90$ & $0.440$ & $0.315$ & $0.300$ & $0.125$ $(0.057,\,0.193)$ & $0.140$ $(0.070,\,0.210)$ \\
$10$ & $4$ & $3$ & 15/15/70 & displaced & $3.60$ & $0.640$ & $0.485$ & $0.490$ & $0.155$ $(0.083,\,0.227)$ & $0.150$ $(0.079,\,0.221)$ \\
$10$ & $4$ & $3$ & 15/15/70 & triangle & $1.49$ & $0.690$ & $0.640$ & $0.640$ & $0.050$ $(-0.007,\,0.107)$ & $0.050$ $(-0.010,\,0.110)$ \\
$10$ & $4$ & $3$ & 15/15/70 & triangle & $2.23$ & $0.980$ & $0.955$ & $0.965$ & $0.025$ $(-0.001,\,0.051)$ & $0.015$ $(-0.007,\,0.037)$ \\
$10$ & $4$ & $3$ & 15/15/70 & triangle & $2.90$ & $1.000$ & $1.000$ & $1.000$ & $0.000$ $(0.000,\,0.000)$ & $0.000$ $(0.000,\,0.000)$ \\
$10$ & $4$ & $3$ & 15/15/70 & triangle & $3.60$ & $1.000$ & $1.000$ & $1.000$ & $0.000$ $(0.000,\,0.000)$ & $0.000$ $(0.000,\,0.000)$ \\
\end{longtable}}

\end{document}